\documentclass[11pt]{amsart}

\usepackage[margin=1in]{geometry} 
\usepackage[english]{babel}

\usepackage{amsmath, amsthm, amssymb, mathtools, mathrsfs, bbm}

\usepackage[ruled,vlined]{algorithm2e}
\usepackage{algorithmic} 
\algsetup{indent=7mm}

\usepackage{graphicx}
\usepackage{epsfig} 
\usepackage{epstopdf} 

\usepackage{booktabs}
\usepackage{multirow}
\usepackage{longtable}
\newcommand{\br}{\toprule}   
\newcommand{\mr}{\midrule}

\usepackage{enumerate}

\usepackage{listings}
\usepackage{xcolor} 

\usepackage{thmtools}
\usepackage{thm-restate}

\usepackage[pdftex, colorlinks=true, linkcolor=blue, citecolor=red, pagebackref=false, pdfborder={0 0 0}]{hyperref}
\usepackage{url}

\usepackage[all]{xy}
\usepackage{tikz-cd}

\usepackage{float}
\usepackage{caption}
\usepackage{subcaption}

\usepackage{relsize}
\usepackage[foot]{amsaddr} 
\usepackage{verbatim}  

\newcommand{\R}{\mathbb{R}}

\newtheorem{thm}{Theorem}[section]
\newtheorem{prop}[thm]{Proposition}
\newtheorem{lem}[thm]{Lemma}
\newtheorem{cor}[thm]{Corollary}

\theoremstyle{definition}

\newtheorem{defn}[thm]{Definition}
\newtheorem{rem}[thm]{Remark}

\theoremstyle{remark}

\numberwithin{equation}{section}

\title{Moments, Time-Inversion and Source Identification for the Heat Equation
}

\date{\today}

\author[K. Liu and E. Zuazua]{Kang Liu$^1$}
\address{$^1$Universit\'e Bourgogne Europe, CNRS, Institut de Mathematiques de Bourgogne, 21000 Dijon, France.}

\author{Enrique Zuazua$^2$$^3$$^4$}
\address{$^2$Friedrich\ -\ Alexander\ -\ Universit\"at Erlangen\ -\ N\"urnberg, Department of Mathematics, Chair for Dynamics, Control, Machine Learning, and Numerics (Alexander von Humboldt Professorship), 91058 Erlangen, Germany.}
\address{$^3$Universidad Aut\'onoma de Madrid, Departamento de Matem\'aticas, 28049 Madrid, Spain.}
\address{$^4$Chair of Computational Mathematics, Fundaci\'on Deusto, 48007 Bilbao, Basque Country, Spain.}
\email{kang.liu@u-bourgogne.fr, enrique.zuazua@fau.de}

\keywords{Heat equation, inverse problem, moment methods, representer theorem, numerical algorithm}

\subjclass[2020]{68T07, 68T09, 90C06, 90C26}

\thanks{E. Zuazua was funded by the European Research Council (ERC) under the European Union's Horizon 2030 research and innovation programme (grant agreement NO: 101096251-CoDeFeL), the Alexander von Humboldt-Professorship program, the ModConFlex Marie Curie Action, HORIZON-MSCA-2021-dN-01, the COST Action MAT-DYNNET, the Transregio 154 Project of the DFG, AFOSR Proposal 24IOE027 and grants PID2020-112617GB-C22/AEI/10.13039/501100011033 and TED2021131390B-I00/ AEI/10.13039/501100011033 of MINECO (Spain), and Madrid GovernmentUAM Agreement for the Excellence of the University Research Staff in the context of the V PRICIT (Regional Programme of Research and Technological Innovation).}

\begin{document}

\begin{abstract}
 We address the initial source identification problem for the heat equation, a notably ill-posed inverse problem characterized by exponential instability. Departing from classical Tikhonov regularization, we propose a novel approach based on moment analysis of the heat flow, transforming the problem into a more stable inverse moment formulation. 
  By evolving the measured terminal time moments backward through their governing ODE system, we recover the moments of the initial distribution. We then reconstruct the source by solving a convex optimization problem that minimizes the total variation of a measure subject to these moment constraints. This formulation naturally promotes sparsity, yielding atomic solutions that are sums of Dirac measures.
  Compared to existing methods, our moment-based approach reduces exponential error growth to polynomial growth with respect to the terminal time.
  We provide explicit error estimates on the recovered initial distributions in terms of moment order, terminal time, and measurement errors. In addition, we develop efficient numerical discretization schemes and demonstrate significant stability improvements of our approach through comprehensive numerical experiments.
\end{abstract}

\maketitle

\section{Introduction}

\subsection{Problem statement and motivation}
Let $d\in\mathbb{Z}_{+}$ (the set of non-negative integers) and consider the heat equation on $\mathbb{R}^d$:
\begin{equation}\label{eq:heat}
\left\{
\begin{array}{ll}
\partial_t u(x,t)-\Delta u(x,t)=0, & (x,t)\in \mathbb{R}^d\times(0,\infty),\\[1mm]
u(\cdot,0)=u_0^\ast,
\end{array}
\right.
\end{equation}
where 
where \(u_0^*\) lies in the tempered distribution space $\mathcal{S}'(\R^d)$.  
Fix a terminal time \(T > 0\). 

The \textbf{time-inversion} or \textbf{initial source identification problem} associated with equation \eqref{eq:heat} is to determine the initial condition \(u_0^*\) from the observed terminal solution \(u(\cdot, T)\). This inverse problem has significant real-world applications, particularly in detecting pollution sources \cite{el2005identification,li2006determining,ozisik2018inverse} and in image denoising \cite{engl1996regularization}.

Although the general formulation considers \(u_0^* \in \mathcal{S}'(\R^d)\), in many practical settings the source is at a finite number of points (e.g.\ pollution spills, point-heat sources):
\begin{equation}\label{eq:Dirac}
    u_0^* = \sum_{i=1}^N m_i \, \delta_{x_i}, 
\end{equation}
with \( m_i \in \mathbb{R} \) and \( x_i \in \mathbb{R}^d \).  
This \emph{atomic ansatz} captures the inherent sparsity of the physical sources and, thanks to the density of finite Dirac sums in $\mathcal S'(\R^d)$ (in the weak-$*$ sense), does not sacrifice generality. Several studies, including \cite{li2014heat,biccari2023two}, focus specifically on this setting and highlight the significant challenges associated with recovering such singular sources. This is  due to the well-known dramatic ill-posedness of the backward resolution of the heat equation. 
Our main result (Theorem~\ref{thm:rate}) requires only that \(u_0^*\) be a finitely supported Radon measure, thus encompassing the Dirac sum case \eqref{eq:Dirac}. This broadens the scope of existing results and improves upon the assumptions commonly found in the literature.

\subsubsection{Ill-posedness.} The solution of the heat equation \eqref{eq:heat} is given by the convolution 
of the heat kernel (denoted by \(G(t)\), see Eq.~\eqref{eq:kernel}) and the initial distribution.

 A natural approach to the initial source identification problem is to solve the optimization problem
\begin{equation}\label{pb:OC_pde}
    \inf_{u_0\in \mathcal{S}'(\R^d)} \left\| G(T)*u_0  - u_{\mathrm{obs}} \right\|_{L^2(\R^d)},
\end{equation}
where \(u_{\mathrm{obs}} \approx u(\cdot,T)\) is the observed terminal distribution (assumed, to fix ideas, to lie in \(L^2(\R^d)\)). However, it is well known that \eqref{pb:OC_pde} is numerically challenging due to its inherent ill-posedness \cite[Sec.~1.5]{engl1996regularization}.

To illustrate the ill-posedness from a Fourier perspective, assume that \eqref{pb:OC_pde} admits an exact solution \(u_0\) so that
\[
G(T)*u_0 = u_{\mathrm{obs}}.
\]
Taking the Fourier transform and using the fact that the Fourier transform of $G(T)$ is $ e^{-T \|\xi\|^2}$,
we deduce that
\begin{equation}
    \label{eq:fourier}
    \widehat{u_0}(\xi) = e^{T\|\xi\|^2}\,\widehat{u_{\mathrm{obs}}}(\xi).
\end{equation}
Since the multiplier \(e^{T\|\xi\|^2}\) grows exponentially with \(\|\xi\|\) and the terminal time $T$, even a small perturbation in the observed data \(u_{\mathrm{obs}}\) is greatly amplified when reconstructing \(u_0^*\). This exponential amplification of high-frequency errors demonstrates the ill-posedness of \eqref{pb:OC_pde}.

To mitigate this ill-posedness, we propose a moment-based method for the initial source identification problem.

\subsection{Moment method}
We replace the heat equation with a surrogate dynamically stable system, that preserves the essence of the evolutionary behavior of the original model while enabling more robust inversion.

We consider the dynamics of the moments of \( u(\cdot, t) \), defined by
\[
M_{\alpha}(t) = \int_{\mathbb{R}^d} x^{\alpha}\, u(x,t) \, dx, \quad \mathrm{for}\, t > 0,
\]
where \(\alpha =(\alpha_1,\ldots,\alpha_d) \in \mathbb{Z}_+^d\) and \(x^{\alpha} = x_1^{\alpha_1} \cdots x_d^{\alpha_d}\). In \cite{duoandikoetxea1992moments}, the authors show that for different initial distributions sharing the same moments up to a finite order, the difference between the corresponding solutions to \eqref{eq:heat} decays faster (as \( t \to \infty \)) than the individual solutions themselves.

Motivated by this result, we propose to recover \( u_0^* \) from the moments of the terminal distribution. 

Fix \( k \in \mathbb{Z}_+ \), the order of the moments under consideration. Given observations \( y_{\alpha} \approx M_{\alpha}(T) \) for all multi-indices \(\alpha\) with \(\|\alpha\|_1 \leq k\), our goal is to reconstruct an accurate approximation of the initial condition \( u_0^* \). The effective means of  numerically observing  $y_{\alpha}$ are presented later in the next subsection. 

To implement this methodology, we propose a natural two-step strategy to reconstruct the unknown initial condition \( u_0^* \) -- an accurate approximation in practice -- from the given measurements \( \{y_{\alpha}\}_{\|\alpha\|_1 \leq k} \):

\begin{enumerate}    \item \textbf{Moment Estimation:}
    Compute the initial moments \( \{M_{\alpha}(0)\}_{\|\alpha\|_1 \leq k} \) from the terminal observations \( \{y_{\alpha}\}_{\|\alpha\|_1 \leq k} \):
    \begin{equation*}
        M_{\alpha}(T) \approx y_{\alpha} \quad \Rightarrow \quad M_{\alpha}(0),
    \end{equation*}
by inverting the dynamics of the moments.

    \item \textbf{Initial Condition Reconstruction:}
    Recover the atomic initial datum \( u_0^* \) from the estimated moments \( \{M_{\alpha}(0)\} \):
    \begin{equation*}
        \{M_{\alpha}(0)\} \quad \Rightarrow \quad u_0 \approx u_0^*,
    \end{equation*}
    through a convex optimization problem.
\end{enumerate}

In the following paragraphs, we detail these two main steps of the methodology and highlight their advantages over direct approaches to the original inverse problem.

\subsubsection{Inverting the dynamics of the moments.}The moments \( M(t) = (M_\alpha(t))_{\|\alpha\|_1 \leq k} \) satisfy a finite-dimensional system of linear ordinary differential equations (ODE) driven by a constant coefficient matrix $A$ of dimension $\binom{k + d}{d} \times \binom{k + d}{d}$. To simplify the notation, we will omit the index $k$ in the notation of $A$. Of course, $\binom{k + d}{d}$ coincides with the number of moments corresponding to $\|\alpha\|_1 \leq k$.

More precisely, let \( e_i \) denote the \( i \)-th canonical basis vector of \( \mathbb{R}^d \). Then, the ODE governing \( M(t) \) is given by (for a more rigorous formulation, see Lemma \ref{lem:ODE}): 
\begin{equation}\label{eq:ODE_moments}
  \frac{d M(t)}{dt} = A \, M(t), \quad \mathrm{where} \ A_{\alpha,\alpha'} = 
  \left\{
    \begin{array}{ll}
      \alpha_i(\alpha_i - 1), & \mathrm{if} \ \alpha' = \alpha - 2 e_i, \\[0.6em]
      0, & \mathrm{otherwise},
    \end{array}
  \right.
\end{equation}
where $A_{\alpha,\alpha'}$ denotes the $(\alpha,\alpha')$-entry of the matrix $A$.
An important property of the matrix \( A \), which is independent of the specific moments of the data under consideration and depends only on the pair $(k,d)$, is that it is \textit{nilpotent} (see Lemma~\ref{lem:growth}), namely:  
\begin{equation*}
    A^{\lfloor k/2 \rfloor +1} = 0.
\end{equation*}
As a consequence, this leads to a polynomial dynamics fulfilling
\begin{equation}\label{eq:moments_grow}
    \|M(0)\|_{\infty}\leq  \|M(T)\|_{\infty}\, \sum_{i=1}^{\lfloor k/2 \rfloor} c_i \, T^{i} ,
\end{equation}
for some constants \( c_i \) independent of \( T \) and only depending on $A$. Here and in the sequel \(\|\cdot\|_{\infty}\) stands for the induced matrix norm on \(\ell^{\infty}\). We refer to Lemma~\ref{lem:growth} for the precise formula of the constants $c_i$.

Comparing \eqref{eq:moments_grow} with \eqref{eq:fourier}, we observe a fundamental difference in error growth: 
the moment inversion error grows at most \emph{polynomially} with $T$, 
in  contrast to the \emph{exponential} error growth in \eqref{eq:fourier}. 
This suggests that backwards solving the ODE system \eqref{eq:ODE_moments} may  offer significantly greater numerical stability than directly addressing the original problem \eqref{pb:OC_pde}.

This is the starting point of our method. But, as mentioned before, a second key step that we describe below, consists of extracting the atomic initial measure out of the recovered moments.

\subsubsection{Recovering initial distribution from its moments.} The reconstructed moments $\{M_{\alpha}(0)\}_{|\alpha| \leq k}$ characterize the initial distribution $u_0$ only up to order $k$, leaving the full recovery under-determined, since infinitely many distributions share these moments.

To determine the sparse initial datum  (as in \eqref{eq:Dirac}), we adopt the \emph{minimum total variation principle}: among all distributions matching the given moments, we select the one with minimal total variation norm.

Furthermore, we restrict the support of the initial data under consideration to a compact domain $\Omega \subseteq \mathbb{R}^d$, which acts as an a priori estimate of $\mathrm{supp}(u_0^*)$. This compactness assumption serves two key purposes:
\begin{itemize}
    \item It prevents dispersion towards infinity and non-physical solutions with unbounded support;
    \item It enables the stable numerical discretization.
\end{itemize}

The resulting optimization problem to recover an efficient approximation of $u_0^*$ writes
\begin{equation}\label{pb:opt_moments}
    \inf_{u_0 \in \mathcal{M}(\Omega)} \|u_0\|_{\mathrm{TV}} 
    \quad \mathrm{subject\, to}\;  
        \int_{\mathbb{R}^d} x^{\alpha}  d\, u_0(x) = M_{\alpha}(0),  \quad \mathrm{for}\, \|\alpha\|_1 \leq k.
\end{equation}

As mentioned above, $\Omega$ is a hyper-parameter of this optimization problem, which plays the role of a priori bound for the support of the true initial distribution $u_0^*$. By \( \mathcal{M}(\Omega) \) we denote the space of Radon measures on \( \Omega \). 

For the existence of a solution of \eqref{pb:opt_moments}, it suffices that \( \Omega \) has non-empty interior (see Theorem~\ref{thm:existence}).  
For quantitative convergence to the true source \( u_0^* \), we take \( \Omega \) to be a hypercube containing the support of \( u_0^* \) (see Theorem~\ref{thm:rate}).

\subsubsection{Overall formulation.} \quad  By combining \eqref{eq:ODE_moments} with the moment-constrained optimization problem \eqref{pb:opt_moments}, we derive the following compact formulation of the moment reconstruction problem: Fix the terminal time \( T \), observe the moments of the measure solution at time \( T \), and then solve the optimization problem
\begin{equation}\label{pb:OC}\tag{P$_{\Omega,k}$}
   \inf_{u_0 \in \mathcal{M}(\Omega)} \|u_0\|_{\mathrm{TV}} 
  \quad \mathrm{subject\ to} \quad  
    \int_{\mathbb{R}^d} x^{\alpha} \, \mathrm{d}u_0(x) = (e^{-TA}\, \mathbf{y})_{\alpha}, 
    \quad \mathrm{for} \ \|\alpha\|_1 \leq k.
\end{equation}
Here, \(A\) is defined in \eqref{eq:ODE_moments} and \(\mathbf{y} = (y_{\alpha})_{\|\alpha\|_1\leq k}\) are observations of moments up to order $k$ at time $T$. Problem \eqref{pb:OC} depends crucially on two key parameters:
\begin{itemize}
    \item The compact domain \( \Omega \subseteq \mathbb{R}^d \) for support restriction:  
This parameter is independent of the terminal time \( T \) and the moment order \( k \). In our main result, \( \Omega \) is fixed as the hypercube \( [-R, R]^d \) aimed to contain the support of \( u_0^* \). When no prior information about the support is available, one can begin by selecting a sufficiently large \( R \)  and gradually decrease \( R \)   until a sharp increase on the minimum value of \eqref{pb:OC} is observed, indicating that the support of \( u_0^* \) has likely been well captured.
 
    \item The moment order $k \in \mathbb{Z}_+$ governing approximation accuracy: The optimal choice of this parameter depends on the observation error; see Remark~\ref{rem:moment_opt} for an explicit asymptotic order. If the error level cannot be determined a priori, one practical strategy is to gradually increase the value of  \( k \) starting from $k=0$, and monitor the total variation of the solution. When a sharp increase occurs, it typically signals that the moment information is no longer reliable beyond that threshold of $k$.

\end{itemize}

Besides the stability benefits discussed earlier, another important advantage of proceeding by solving  \eqref{pb:OC} is that it admits a solution in the form of a finite sum of Dirac measures \eqref{eq:Dirac}; see Theorem~\ref{thm:existence}.  This follows from classical results known as \textit{Representer Theorems} \cite{fisher1975spline}.

The process for tackling  \eqref{pb:OC}, which consists first in inverting \eqref{eq:ODE_moments} and then solving \eqref{pb:opt_moments}, as described earlier, is shown in the following diagram for clarity.
  \begin{center}
\begin{tikzpicture}[->, >=Stealth, auto]
    \node (A) at (0,0) {$y_{\alpha}$};
    \node (B) at (2.5,2.5) {$M_{\alpha}(0)$};
    \node (C) at (5,0) {$u_0 $};
    
    \draw[->, thick] (A) -- node[left] {Inverting \eqref{eq:ODE_moments}} (B);
    \draw[->, thick] (B) -- node[right] {Solving \eqref{pb:opt_moments}} (C);
    \draw[->] (A) to node[below] {Solving \eqref{pb:OC}} (C);
    
 \node[left=3mm] at (A) { $M_{\alpha}(T)\, \approx  $};
    \node[right=3mm] at (C) {$\approx \, u_0^{*}$};
\end{tikzpicture}      
  \end{center} 

\subsection{Main results}  
\subsubsection{Main contributions.}Our main contributions to the theoretical analysis of problem \eqref{pb:OC} are summarized as follows:
\begin{enumerate}\setlength{\itemsep}{5pt} 
    \item \textbf{(Existence of atomic solutions)} In Theorem \ref{thm:existence} we first establish the existence of an atomic solution to \eqref{pb:OC} of the form \begin{equation}\label{intro_eq:extreme_point}
        u^k_0 = \sum_{i=1}^{\binom{k + d}{d}} m_i \delta_{x_i}.
    \end{equation}
by leveraging results from representer theorems. The number of masses in the sparse atomic reconstruction is exactly the same as the number of moments under consideration. This existence result holds for any observed data \( \mathbf{y} = (y_{\alpha})_{\|\alpha\|_1\leq k} \).
    
    \item \textbf{(Error estimate)} Next, in Theorem \ref{thm:rate}, assuming that \( u_0^* \) is compactly supported within \( \Omega = [-R, R]^d \), we establish an error estimate that quantifies the difference between the recovered solution \( u_0^k \) and the true initial distribution \( u_0^* \). This difference, measured by the Kantorovich distance (see Definition \ref{def:kant}), is given by:
    \begin{equation}\label{intro_eq:conv}
       \| u_0^{k}- u_0^{*}\|_{\mathrm{Kant}} \leq C_1 \left(\frac{  \|u_0^*\|_{\mathrm{TV}}}{k} +  k^{k/2}\,e^{C_2\,k} \, \max\left\{ T^{\lfloor k/2\rfloor} \, , \,1 \right\} \,\|\epsilon\|_{\infty}\right),
    \end{equation}
    where \( C_1 \) and \(C_2\) are constants depending only on $d$ and $R$.  The variable \( \epsilon = (\epsilon_{\alpha})_{\|\alpha\|_1\leq k} \) represents the observation error of the moments of \( u(\cdot, T) \). 
The right-hand side of \eqref{intro_eq:conv} decomposes into two distinct components: 
(1) the first term is independent of $\epsilon$ and decays at the rate $1/k$ as the moment order $k \to \infty$; 
(2) the second term grows with $k$ (analogous to the frequency cutoff in Fourier analysis), but for any fixed $k$ it exhibits only polynomial growth in $T$, in sharp contrast to the \emph{exponential growth} appearing in \eqref{eq:fourier}.  
For a detailed discussion of the constants involved, as well as an optimal choice of $k$ balancing these two contributions, we refer to Remarks~\ref{rem:conv} and~\ref{rem:moment_opt}.
\end{enumerate}

Finally, we emphasize that, in the large-time regime, if the observational noise becomes comparable to or larger than the amplitude of the heat solution at time $T$, the moment errors necessarily dominate. In this case, the second term in \eqref{intro_eq:conv} becomes large, significantly degrading the quality of the reconstructed solution.

\subsubsection{Techniques used in the proof.}The proof of our main results combines techniques from several areas, including the representer theorem, optimal transportation, and polynomial approximation.

We apply the representer theorem from \cite{fisher1975spline} to show that the extreme points of the solution set of \eqref{pb:OC} are sums of Dirac measures. Optimization problems over measures, together with their associated representer theorems, play a fundamental role in classical inverse problems; see \cite{unser2017splines} and references therein. 
The core idea behind representer theorems is that the extreme points of the feasible set, defined by linearly independent test functions, take the form of finite sums of Dirac measures. In practice, when the test functions are monomials, this recovers the classical result known as Tchakaloff’s Theorem \cite{bayer2006tchakaloff}.

Notably, problem \eqref{pb:OC} also bears a strong resemblance to mean-field relaxation problems arising in the training of neural networks \cite{chizat2018global,mei2018mean,liu2025representation}. The use of representer theorems in such machine learning settings has recently drawn increasing attention; see, for example, \cite{bach2017breaking,unser2019representer,liu2025representation}.

To quantify the distance between the solution of \eqref{pb:OC} and the true initial distribution \( u_0^* \), we use the Kantorovich distance, as introduced in \cite{hanin1999kantorovich,piccoli2016properties}, which is also referred to as the flat metric in the literature. This metric generalizes the Wasserstein distance \cite[Sec.~6]{villani2009} from probability measures to signed measures. We discuss the connection between the Kantorovich and Wasserstein distances in Section~\ref{sec_wasserstein}. A key feature of the Kantorovich distance is that it is defined via duality with respect to test functions that are \(1\)-Lipschitz continuous, i.e.,
there exists $v^*$, with $\|v^*\|_{\mathcal{C}(\Omega)}\leq 1$ and Lip$(v^*)\leq 1$, such that 
    \begin{equation*}
         \| u_0^{k}- u_0^*\|_{\mathrm{Kant}} = \int_{\Omega} v^*\, d (u_0^{k}- u_0^*).  
    \end{equation*}
The test function \( v^* \) can be approximated by a degree-\( k \) polynomial \( p^k \) with an explicit convergence rate of \( 1/k \), as guaranteed by Jackson's Theorem \cite{newman1964jackson}. This allows us to decompose the integral into the sum of two terms:
\begin{enumerate}
    \item First term: \( \int_{\Omega} (v^* - p^k) \, d(u_0^k - u_0^*) \). Since the integrand is uniformly bounded by \( C/k \), it suffices to estimate the total variation of \( u_0^k - u_0^* \), which is controlled via Lemma~\ref{prop:a priori bound}.
    
    \item Second term: \( \int_{\Omega} p^k \, d(u_0^k - u_0^*) \). Since \( p^k \) is a polynomial of degree \( k \), its integral against \( u_0^k - u_0^* \) depends on two components: the coefficients of \( p^k \) and the difference between the moments of \( u_0^k \) and \( u_0^* \) up to order \( k \). The moment differences are controlled by the backward stability of the moment system; see \eqref{eq:moments_grow}. The coefficients of \( p^k \) are in turn governed by its \( L^\infty \)-norm in $\Omega$, via Bernstein's inequality; see Lemma~\ref{lem:poly_coef}.
\end{enumerate}

\subsection{Numerical implementation}
For the numerical solution of the initial source identification problem, we must first approximate the moments numerically and then design an appropriate algorithm to solve \eqref{pb:OC}.

The moments \(M_\alpha(T)\) are not directly observable; only pointwise measurements of \(u(\cdot,T)\) at sensor locations are available.  
    We approximate the moments using quadrature rules,
    \[
        M_\alpha(T) \approx y_\alpha = \sum_{j=1}^n w_j\, z_j^\alpha\, u(z_j,T),
    \]
    where \( (z_j,w_j)_{j=1}^n \) denote the quadrature nodes and weights.
    In our analysis, we assume that the sensor locations coincide with the quadrature nodes.
    In Section~\ref{sec:quadrature}, we compare two quadrature strategies: uniform grids and Gauss--Hermite nodes. Numerical results show that Gauss--Hermite quadrature performs significantly better, especially for large $T$. Moreover, in Remark~\ref{rem:total-error}, we provide an overall error estimate that accounts for pointwise observational noise at the Gauss--Hermite nodes.

Once the moments $y_{\alpha}$ are obtained using the quadrature formulas above, we proceed to solve \eqref{pb:OC}, a convex optimization problem posed in the infinite-dimensional space $\mathcal{M}(\Omega)$.
We adopt a \emph{discretize--then--optimize} approach. The domain \(\Omega\) is discretized into points \(\Omega_h=\{x_1,\dots,x_N\}\), and \eqref{pb:OC} is replaced by the finite-dimensional convex program
\begin{equation}\label{intro_pb:OC_dis}
    \inf_{\mathbf{m}\in \mathbb{R}^N} \|\mathbf{m}\|_{\ell^1} 
    \quad
    \mathrm{subject\; to}
    \quad
    B\, \mathbf{m} \;=\; e^{-T A}\, \mathbf{y},
\end{equation}
    where \(B\in \R\) is the rectangular Vandermonde matrix with entries \(b_{\alpha,i}=x_i^\alpha\).  
    Under mild assumptions, solutions exist (Lemma~\ref{lem:exist}) and can be computed efficiently using the \textit{Simplex} method.  
    
\color{black}

\subsection{Related work}
\subsubsection{Tikhonov regularization methods.}
Tikhonov regularization \cite{tikhonov1963solution} is a classical approach to mitigating the ill-posedness of the initial source identification problem. It adds a regularization term to \eqref{pb:OC_pde} to penalize high-frequency components of the solution.

In \cite{biccari2023two}, the authors introduce a combined penalty involving the \( L^1 \) and \( L^2 \) norms of \( u_0 \) in the formulation of \eqref{pb:OC_pde}. However, this penalization prevents the solution from being a sum of Dirac measures as in \eqref{eq:Dirac}. To address this limitation, they propose a two-stage method: first, solve the penalized optimization problem; then, detect the local maxima of the solution; and finally, determine the amplitudes of these maxima by solving a finite-dimensional convex problem. Although numerically efficient for small terminal times (e.g., \( T = 0.01 \) in \cite{biccari2023two}), this approach lacks reliable theoretical guarantees of convergence and is unstable for large $T$.

In contrast, \cite{casas2019using} replaces the \( L^p \) penalty with a TV norm penalty \( \|u_0\|_{\mathrm{TV}} \) for \eqref{pb:OC_pde}. The authors analyze the structure of the solution using first-order optimality conditions \cite[Thm.~3.1 and Thm.~4.4]{casas2019using}. When the minimizers of the adjoint solution are finite, this method recovers a sum of Dirac measures as in \eqref{eq:Dirac}. In our work, such sparsity arises naturally through the representer theorem and the finite observations of moments. The regularization formulations in \cite{casas2019using} are highly sensitive to the choice of penalty coefficients, and no efficient numerical algorithms are designed to solve these problems.

A further limitation of these Tikhonov-based methods on \eqref{pb:OC_pde} is their reliance on full knowledge of the terminal solution \( u(\cdot, T) \), which is unrealistic in practical scenarios where only discrete observations of \( u(x,T) \) at sensor locations are available. Our method addresses this by computing moments through quadrature of these pointwise values. The issue of limited observations is also noted in \cite{li2014heat}, where the authors adapt the formulation of \eqref{pb:OC_pde} using value residuals at sensor locations and impose an $L^1$ penalty. They prove convergence with respect to observation error (see \cite[Thm.~2.2]{li2014heat}), but the result holds only in one spatial dimension, assumes the true initial distribution \( u_0^* \) is a single Dirac measure, and does not explicitly characterize the dependence of the convergence rate on the terminal time \( T \). Our main result, Theorem~\ref{thm:rate}, improves on these limitations by providing an explicit error bound that scales polynomially with \( T \). In contrast to the fixed small terminal times used in \cite{biccari2023two} and \cite{li2014heat} (both with \( T = 0.01 \)), our method remains stable for terminal times up to \( T = 100 \) in two dimensions, offering robust performance across a much broader range of scenarios.
  
\subsubsection{Algorithms for the discretized problem.}
The optimization problem \eqref{intro_pb:OC_dis} is convex, and a variety of algorithms can be used to solve it. These include Bregman iteration \cite{osher2005iterative} and its variants \cite{goldstein2009split}, the Alternating Direction Method of Multipliers (ADMM) \cite{glowinski1975approximation}, and primal-dual methods \cite{chambolle2011first}. Following \cite{chen2001atomic}, we rewrite \eqref{intro_pb:OC_dis} as a linear programming problem.
Consequently, a natural and robust approach is to solve the discretized problem \eqref{intro_pb:OC_dis} using the simplex method. Importantly, the simplex method yields solutions located at extreme points of the feasible set (see \cite[Thm.~3.3]{bertsimas1997introduction}), which leads to sparse solutions with fewer support points in the recovered initial distribution.

\subsection{Organization and notations}
\subsubsection{Organization.}In Section~\ref{sec_main}, we establish the existence of solutions to the optimization problem \eqref{pb:OC} and analyze their convergence to the true initial condition 
\( u_0^* \) 
  under appropriate assumptions. Building on these theoretical foundations, Section~\ref{sec_dis} develops numerical methods for both moment observation and efficient solution of \eqref{pb:OC}. The effectiveness of our moment-based approach for initial source identification is then demonstrated through numerical experiments in Section~\ref{sec_num}. Finally, Section~\ref{sec_proof} collects the technical lemmas required for proving our main results.

\subsubsection{Notations.}We begin by fixing notation:
\begin{itemize}
    \item For $x \in \mathbb{R}^n$, $\|x\|_p$ denotes the standard $\ell^p$-norm
    \item $\lfloor a \rfloor$ represents the floor function for $a \in \mathbb{R}$
    \item $\binom{n}{m}$ is the binomial coefficient for $n,m \in \mathbb{N}$
\end{itemize}

Let $\Omega \subset \mathbb{R}^d$ be compact. We consider:
\begin{itemize}
    \item $\mathcal{C}(\Omega)$: The Banach space of continuous functions on $\Omega$ equipped with the supremum norm $\|\cdot\|_{\mathcal{C}(\Omega)}$
    \item $\mathcal{M}(\Omega)$: Its dual space, identified with Radon measures on $\Omega$ having finite total variation.\end{itemize}

\begin{defn}[Kantorovich Norm \cite{hanin1999kantorovich,piccoli2016properties}]\label{def:kant}
    For $u_0 \in \mathcal{M}(\Omega)$, the Kantorovich norm is
    \begin{equation}\label{eq:kant_norm}
        \|u_0\|_{\mathrm{Kant}} \; \triangleq  \max_{\begin{array}{c}
    v \in \mathcal{C}(\Omega),\\ 
    \|v\|_{\mathcal{C}(\Omega)} \leq 1,\\
    \mathrm{Lip}(v) \leq 1
  \end{array}} \int_{\Omega} v \, d\,u_0 ,
\end{equation}
    where $\mathrm{Lip}(v)$ denotes the Lipschitz constant of $v$. The maximum is indeed attained \cite[Eq.\,(4.2)]{hanin1999kantorovich}.
\end{defn}
In this article, we use the Kantorovich norm to quantify the difference between measures. This choice is motivated by the fact that the recovered initial distribution and the true one may have different total masses, in which case the classical Wasserstein distance is not well-defined. When the total mass is preserved, these two norms are equivalent, leading to  an error estimate in the Wasserstein distance, as detailed in Section~\ref{sec_wasserstein}.

\section{Main results}\label{sec_main}
The main results of this article are presented in Section~\ref{sec_theorems}. In Section~\ref{sec_moments}, we analyze the dynamics of moments for solutions to the heat equation; these results serve as preliminary tools for the proof of the main error estimates in Section~\ref{sec_proof_main}. Finally, in Section~\ref{sec_wasserstein}, we discuss the error estimates in the Wasserstein sense.

\subsection{Main results}\label{sec_theorems}
We first show the existence of the solutions of problem \eqref{pb:OC}. 

\begin{thm}[Existence]\label{thm:existence}
Assume that \( \Omega \) is compact and has a non-empty interior. Then, for any \( k \in \mathbb{Z}_+ \), any $T>0$  and any vector \( \mathbf{y} = (y_{\alpha})_{\|\alpha\|_1 \leq k} \in \mathbb{R}^{\binom{k + d}{d}} \), the following statements hold:
\begin{enumerate}
    \item The solution set of \eqref{pb:OC} is non-empty, convex, and compact in the weak-\(*\) topology;
    \item The extreme points of the solution set of \eqref{pb:OC} are of the form:
    \begin{equation}\label{eq:extreme_point}
        u^k_0 = \sum_{i=1}^{\binom{k + d}{d}} m_i \delta_{x_i},
    \end{equation}
    where \( m_i \in \mathbb{R} \), \( x_i \in \Omega \), and \( \delta_{x_i} \) are Dirac measures.
\end{enumerate}
\end{thm}
\begin{proof}
Since \(\Omega\) has nonempty interior, the monomials \(\{x^{\alpha}\}_{\|\alpha\|_1\leq 1}\) are linearly independent in \(\mathcal{C}(\Omega)\). Therefore, the conclusions of Theorem \ref{thm:existence} follow from the Representer Theorem  in \cite{fisher1975spline} (see also Theorem \ref{thm:representation}).
\end{proof}

Next, we bound the error between the solution of \eqref{pb:OC} and the true initial distribution \(u_0^*\) in terms of the moment order \(k\) and the measurement error.  To quantify this discrepancy, we use the Kantorovich norm (see Definition~\ref{def:kant}).

\begin{thm}[Error estimate]\label{thm:rate}
Assume $u_0^*$ is a compactly supported Radon measure in $\R^d$ with a finite total variation and that \((u, u_0^*)\) satisfies the heat equation \eqref{eq:heat}.
Let $R>0$ be such that 
\[\mathrm{supp}(u_0^*) \subseteq \Omega \triangleq [-R,R]^d.
\]
Fix any \( k \in \mathbb{Z}_+^* \) and let
\[
\mathbf{y} = (y_{\alpha})_{\|\alpha\|_1 \leq k} \in \mathbb{R}^d
\]
denote the vector of observed moments of $u(\cdot,T)$. The observation error is given by
\[
\epsilon = (\epsilon_{\alpha})_{\|\alpha\|_1 \leq k}, \quad \mathrm{with} \quad \epsilon_{\alpha} = y_{\alpha} - \int_{\mathbb{R}^d} x^{\alpha}\, u(x,T) \, dx.
\]
Let \( u_0^k \) be any solution of \eqref{pb:OC}. Then, we have
\begin{equation}\label{eq:conv}
\|u_0^k - u_0^*\|_{\mathrm{Kant}} \leq \frac{C_d\, R\, \|u_0^*\|_{\mathrm{TV}}}{k} + C_{d,R} (k)\; \max\{ T^{\lfloor k/2\rfloor},1\} \,\|\epsilon\|_{\infty},
\end{equation}
where \(C_d\) is a constant, depending only on \(d\), and
\begin{equation}\label{grosseconstante}
C_{d,R} (k) =  \left( \sqrt{\frac{k}{\pi}} + \frac{C_d \, R}{ \sqrt{\pi\, k}}\right) \,  \, \exp\Bigg( k\left(1+ 2d/R +  \ln \sqrt{k} \right)\Bigg) .
\end{equation}
\end{thm}

\begin{proof}
    The proof is presented in Section \ref{sec_proof_main}.
\end{proof}

\begin{rem}[Error estimate analysis]\label{rem:conv}
 The right-hand side of the estimate \eqref{eq:conv} consists of two terms:
\begin{equation*}
    \frac{C_d\, R\, \|u_0^*\|_{\mathrm{TV}}}{k}   \quad \mathrm{and} \quad   C_{d,R} (k)\; \max\{ T^{\lfloor k/2\rfloor},1\} \,\|\epsilon\|_{\infty}.
\end{equation*} 
Let us analyze each term:
\begin{enumerate}
    \item The first term is independent of the terminal time \( T \) and the observation error \( \epsilon \). It decays to zero at a rate of \( 1/k \). Therefore, we represent it as
    \begin{equation*}
        \frac{C_0}{k},
    \end{equation*}
    where \( C_0 \) is a constant independent of \( k \), \( T \), and \( \epsilon \).

\item For the second term, recalling the definition of \( C_{d,R}(k) \) from \eqref{grosseconstante}, we obtain the following leading-order expression:
    \begin{equation*}
        C_1\,(T\,k)^{k/2} e^{C_2 k} \|\epsilon\|_{\infty},
    \end{equation*}
    where \( C_1 \) and \( C_2 \) are constants independent of \( k \), \( T \), and \( \epsilon \). 
\begin{itemize}
    \item Fixing \( k \), we observe that the prefactor of \( \|\epsilon\|_{\infty} \) grows polynomially with \( T \), which marks a significant improvement over the classical ill-posedness in time illustrated in \eqref{eq:fourier}, where the error increases exponentially with time. 
    \item The moment order \( k \) plays a role analogous to the frequency amplitude in the Fourier analysis of the inverse problem for \eqref{eq:heat}. For large \( k \), achieving a small reconstruction error at final time requires increasingly accurate moment observations\ -\ that is, a smaller observation error \( \epsilon \). This reflects the inherent ill-posedness of the backward heat equation in the high-frequency regime.
\end{itemize}
\end{enumerate}
\end{rem}

\begin{rem}[Optimal choice of the moment order $k$]\label{rem:moment_opt}
Let the final time be~$T>2$, and assume that the observation errors for all moments (until infinity) satisfy
\begin{equation}\label{eq:constraint}
   \delta_1 \;\le\; \max_{\alpha\in\mathbb{Z}_+^d}|\epsilon_\alpha|
   \;\le\;\delta_2,
   \quad \mathrm{for\, some \,} 0<\delta_1\le\delta_2<1/(2e).
\end{equation}
As noted in Remark~\ref{rem:conv}, the first term in the convergence estimate \eqref{eq:conv} decays in~$k$, whereas the second grows.  Hence the optimal moment order~$k^*$ is determined by balancing
\[
   \frac{C_0}{k}
   \;=\;
   C_1\,(T\,k)^{k/2}\,e^{C_2k}\,\|\epsilon\|_{\infty},
\]
a transcendental equation in~$k$.  Under the constraint \eqref{eq:constraint}, its unique solution satisfies the following bounds:
\begin{equation}\label{eq:optimal_k}
   \frac{c_1\,\ln(1/\delta_2)}%
        {\ln T \;+\;\ln\ln(1/\delta_2)}
   \;\le\; k^*
   \;\le\;
   \frac{c_2\,\ln(1/\delta_1)}%
        {\ln T \;+\;\ln\ln(1/\delta_1)},
\end{equation}
where $c_1,c_2>0$ are constants independent of~$\delta_i$ and~$T$.  Substituting this $k^*$ back into \eqref{eq:conv} yields
\begin{equation}\label{eq:optimal_value}
   \bigl\|u_0^{k^*}-u_0^*\bigr\|_{\mathrm{Kant}}
   \;\le\;
   \frac{c_3\bigl(\ln T + \ln\ln(1/\delta_2)\bigr)}%
        {\ln(1/\delta_2)},
\end{equation}
for some constant $c_3$ also independent of~$\delta_i$ and~$T$.  From these estimates, we observe:
\begin{itemize}
  \item For fixed $T$, as the maximal error bound $\delta_2\to0$, the optimal order $k^*\to\infty$ and the reconstruction error vanishes.
  \item If $T$ increases so that the minimal error bound $\delta_1$ typically grows, then $k^*$ decreases and the reconstruction error increases, in agreement with the numerical results in Table~\ref{tab:noise-vs-T}. 
  In other words, for large final time $T$, achieving an accurate reconstruction requires the observational noise on the moments to be very small. Since the moments are computed numerically via quadrature, this translates into requiring the sensor noise to be much smaller than the amplitude of the heat solution at the sampled points.
\end{itemize}
\end{rem}

\subsection{The dynamic system of moments}\label{sec_moments}
This subsection is dedicated to a rigorous description of the ODE satisfied by the moments of the solution to the heat equation. This ODE plays an essential role in the proof of Theorem~\ref{thm:rate}.

Throughout this subsection, we fix $\Omega$ as a compact subset of $\mathbb{R}^d$. For any initial data \( u_0 \in \mathcal{M}(\Omega) \), the associated solution $u$ of the heat equation is given explicitly by:
\begin{equation}
    \label{eq:heat_u_0}
    u(x,t) = \bigl(G(\cdot,t) * u_0\bigr)(x), 
\quad \mathrm{for\, } (x,t) \in \mathbb{R}^d \times \mathbb{R}_+,
\end{equation}
where \(G(x,t)\) denotes the heat kernel:
\begin{equation}\label{eq:kernel}
    G(x,t) 
    = \frac{1}{(4\pi t)^{d/2}} 
      \exp \Bigl(-\frac{\|x\|^2}{4t}\Bigr),
    \quad (x,t) \in \mathbb{R}^d \times \mathbb{R}_+.
\end{equation}
For any $\alpha\in \mathbb{Z}_+^d $ and $t\geq 0$, define
   \begin{equation*}
  M_{\alpha}(t) \;
  \triangleq
  \left\{
    \begin{array}{ll}
      \displaystyle \int_{\mathbb{R}^d} x^{\alpha} \, u(x,t) \, \mathrm{d}x, & \quad \mathrm{if}\ t > 0, \\[0.6em]
      \displaystyle \int_{\mathbb{R}^d} x^{\alpha} \, \mathrm{d}u_0(x), & \quad \mathrm{if}\ t = 0.
    \end{array}
  \right.
\end{equation*}
The following holds:

\begin{lem}[The ODE system of moments]\label{lem:ODE}
    Fix any \( k \in \mathbb{Z}_+ \) and define
\[
M(t) \triangleq \left( M_\alpha(t) \right)_{\|\alpha\|_1 \leq k}, \quad t \geq 0.
\]
Then, we have
\begin{equation}\label{eq:ODE}
  \left\{
    \begin{array}{ll}
      \displaystyle \frac{d M(t)}{d t} = A \, M(t), & \quad \mathrm{for}\ t > 0, \\[0.6em]
      M(0) = \displaystyle \lim_{t \to 0^+} M(t), &
    \end{array}
  \right.
\end{equation}
with \( A \) as in \eqref{eq:ODE_moments}.
\end{lem}

\begin{proof} We proceed in several steps.

\medskip
\noindent\textbf{Step 1} (Case $t>0$).   For any \( t > 0 \), by the heat equation, we have
\[
\frac{d M_{\alpha}(t)}{dt} = \int_{\mathbb{R}^d} x^{\alpha} \partial_t u (x,t) \,dx = \int_{\mathbb{R}^d} x^{\alpha} \Delta u (x,t) \,dx.
\]
Since \( u(\cdot, t) \in \mathcal{S}(\mathbb{R}^d) \), applying integration by parts twice, we obtain
\[
\int_{\mathbb{R}^d} x^{\alpha} \Delta u (x,t) \,dx = \sum_{i=1}^d \alpha_i (\alpha_i-1) \int_{\mathbb{R}^d} x^{\alpha - 2 e_i} u(x,t) \,dx,
\]
where \( x^{\alpha - 2e_i} \) is set to be zero if \( \alpha_i \leq 1 \) for any \( i \). Therefore, equation \eqref{eq:ODE} holds for \( t > 0 \).

\medskip
\noindent\textbf{Step 2} (The limit $t\to 0^+$). 
 For any $t>0$, we have
\begin{equation*}
\begin{array}{rl}
  M_{\alpha}(t) 
  = \displaystyle \int_{\mathbb{R}^d} z^{\alpha}\, u(z,t)\, \mathrm{d}z & = \displaystyle \int_{\mathbb{R}^d} z^{\alpha}\, (G(\cdot,t) * u_0)(z)\, \mathrm{d}z \\[0.6em]
  &= \displaystyle \int_{\mathbb{R}^d} \int_{\Omega} z^{\alpha}\, G(x - z,t)\, \mathrm{d}u_0(x)\, \mathrm{d}z.
\end{array}
\end{equation*}
Note that there exists a constant $C_{\alpha}>0$ such that
    \begin{equation}\label{eq:claim}
         \int_{\R^d} |z|^{\alpha} G(x-z, t) \, d z  \leq  C_{\alpha} (|x|^{\alpha} +t^{\|\alpha\|_1/2}), \; \forall\,(x,t)\in \R^d\times (0, +\infty).
    \end{equation}
Indeed, using \eqref{eq:claim} and Fubini's theorem, we can interchange the order of integration in \( M_{\alpha}(t) \), yielding
\[
M_{\alpha}(t) = \int_{\Omega} \int_{\mathbb{R}^d} z^{\alpha} G(x-z, t) \,dz\,du_0(x).
\]
The inner integral can be computed explicitly:
\[
\int_{\mathbb{R}^d} z^\alpha\,G(x-z,t)\,\mathrm{d}z =
\sum_{\begin{array}{c} \beta \le \alpha, \\ \beta_i\ \mathrm{even\ for\ all}\ i \end{array}} 
\binom{\alpha}{\beta} \, x^{\alpha - \beta} \, (-1)^{\|\beta\|_1} 
(2t)^{\|\beta\|_1/2} \prod_{i=1}^d (\beta_i - 1)!!.
\]
Taking the limit as \( t \to 0^+ \), we obtain
\[
\lim_{t \to 0^+} \int_{\mathbb{R}^d} z^{\alpha} G(x - z, t) \,\mathrm{d}z = x^{\alpha}, 
\quad \mathrm{for\ all}\ x \in \mathbb{R}^d.
\]
By inequality \eqref{eq:claim}, the absolute value of the integrand is uniformly 
integrable with respect to \( u_0 \) for bounded \( t \ge 0 \), since \( u_0 \) has 
compact support. Therefore, by Lebesgue's Dominated Convergence Theorem, we conclude that
\[
\lim_{t \to 0^+} M_{\alpha}(t) = M_{\alpha}(0).
\]

\medskip
\noindent\textbf{Step 3} (Proof of \eqref{eq:claim}). Let us rewrite the convolution in \eqref{eq:claim} as
\[
\int_{\mathbb{R}^d} |z|^{\alpha} G(x - z, t) \, dz = \int_{\mathbb{R}^d} |x - z|^{\alpha} G(z, t) \, dz.
\]
   Decompose the right-hand-side by the following inequality:
    \begin{equation*}
        |x-z|^{\alpha} \leq 2^{\|\alpha\|_1} (|x|^{\alpha} + |z|^{\alpha}).
    \end{equation*}
   Since the total mass of \( G(\cdot, t) \) is 1, it remains to estimate the second integral involving \( |z|^{\alpha} \). By a direct computation, splitting the integral along each coordinate direction, we obtain
    \begin{equation*}
      \int_{\R^d} |z|^{\alpha} G(z,t)\, dz = \prod_{i=1}^d \frac{1}{\sqrt{4\pi t}} \int_{\R} |z_i|^{\alpha_i} \, \exp \left(-\frac{z_i^2}{4t}\right)\, dz_i  = \prod_{i=1}^d \frac{(4t)^{\frac{\alpha_i}{2}}}{\sqrt{\pi}} \Gamma \left(\frac{\alpha_i+1}{2}\right),
    \end{equation*}
    where $\Gamma$ is the Gamma function. This yields the term involving $t^{\|\alpha\|_1/2}$ in \eqref{eq:claim}. 
\end{proof}

\begin{lem}[Growth rate]\label{lem:growth}
    The matrix \(A\) defined by \eqref{eq:ODE_moments} is \( (\lfloor k/2 \rfloor +1) \)-nilpotent, i.e.,
    \begin{equation}\label{eq:A0}
        A^{\lfloor k/2 \rfloor +1} = 0.
    \end{equation}
    As a consequence, the solution $M(\cdot)$ of \eqref{eq:ODE} satisfies the following:
\begin{equation}\label{eq:grownwall}
    \| M(0) \|_{\infty} \leq  
       \| M(t) \|_{\infty} \, \sum_{j=0}^{\lfloor k/2 \rfloor} \frac{k^j(k-1)^j}{j\, ! } t^j, \quad\forall\, t\geq 0.
\end{equation}
\end{lem}

\begin{proof}
Let $\mathcal{P}_k(\mathbb{R}^d)$ be
the space of polynomials in \(d\) variables of total degree at most \(k\). For any monomial \(x^\alpha\) with \(\alpha=(\alpha_1,\dots,\alpha_d)\), we have 
\[
\Delta x^\alpha = \sum_{i\in I_{\alpha}} \alpha_i(\alpha_i-1)x^{\alpha-2e_i}, \quad \mathrm{where }\,I_{\alpha} = \left\{i\in\{1,\ldots, d\} \,\mid\, \alpha_i\geq 2 \right\}.
\]
Thus, if we define a linear operator
\[
L \colon  \mathcal{P}_k(\mathbb{R}^d) \to \mathcal{P}_k(\mathbb{R}^d), \quad L\, p = \Delta\, p,
\]
its matrix representation in the basis \(\{x^\alpha\}_{\|\alpha\|_1\le k}\) is exactly
\[
A_{\alpha,\alpha'} =
\left\{
  \begin{array}{ll}
    \alpha_i(\alpha_i - 1), & \quad \mathrm{if}\ \alpha' = \alpha - 2e_i, \\[0.6em]
    0, & \quad \mathrm{otherwise}.
  \end{array}
\right.
\]

Note that applying \(\Delta\) reduces the total degree of a polynomial by 2. 
Thus, after applying the Laplacian \(\lfloor k/2 \rfloor +1 \) times, we obtain
\[
\Delta^{\lfloor k/2 \rfloor +1} \, p  = 0, \quad \forall\, p\in \mathcal{P}_k(\mathbb{R}^d).
\]
Since \(A\) is the matrix representation of \(\Delta\), equation \eqref{eq:A0} follows.

Next, from \eqref{eq:ODE} we deduce that
\[
M(0) = e^{-tA}\, M(t),\quad \forall\, t\ge 0.
\]
Hence,
\[
\|M(0)\|_{\infty} \le \|M(t)\|_{\infty}\, \|e^{-tA}\|_{\infty}.
\]
Since \(A\) is \(\bigl(\lfloor k/2 \rfloor + 1\bigr)\)-nilpotent, we have
\[
e^{-tA} = \sum_{j=0}^{\lfloor k/2 \rfloor} \frac{(-A)^j}{j!}\, t^j,\quad \forall\, t\ge 0.
\]
Moreover, by definition,
\[
\|-A\|_{\infty} = \max_{\|\alpha\|_1\le k} \sum_{\alpha'} |A_{\alpha,\alpha'}| = \max_{\|\alpha\|_1\le k} \sum_{i=1}^d \bigl(\alpha_i(\alpha_i-1)\bigr)_+ = k(k-1).
\]
Combining these estimates, we obtain the desired inequality \eqref{eq:grownwall}.
\end{proof}

\begin{lem}[A priori bound]\label{prop:a priori bound}
  Let $ \Omega = [-R,R]^d$ for some $R>0$.
   For any \( k \in \mathbb{Z}_+ \) and any vector \( \mathbf{y}= (y_{\alpha})_{\|\alpha\|_1 \leq k} \), let $u_0^k$ be a solution of \eqref{pb:OC}. Then, we have
    \begin{equation*}
    \|u_0^k\|_{\mathrm{TV}} \leq \| \mathbf{y} \|_{\infty}  \,    \sqrt{\frac{k}{\pi}} \, \exp\Bigg( k\left(1+ 2d/R +  \ln \sqrt{k} \right)\Bigg) 
\max \left\{ T^{\lfloor k/2 \rfloor}, \, 1 \right\}.
\end{equation*}
\end{lem}

\begin{proof}
   Let \(u^k_0\) be a solution of \eqref{pb:OC}, and let \(u\) be its corresponding heat flow. By Lemma \ref{lem:ODE}, the moments of \(u\) satisfy the ODE system \eqref{eq:ODE}, and their growth rates are controlled by \eqref{eq:grownwall}. Since the moments at time \(T\) are given by \(\{y_\alpha\}_{\|\alpha\|_1\leq k}\), it follows from \eqref{eq:grownwall} that
\[
\left|\int_{\mathbb{R}^d} x^\alpha \, d u^k_0(x)\right|
\;\le\;
\|\mathbf{y}\|_{\infty} \sum_{j=0}^{\lfloor k/2 \rfloor} \frac{k^j(k-1)^j}{j\, ! } T^j, \quad \forall\,  \|\alpha\|_1 \le k.
\]
By Lemma \ref{lem:Stirling}, we have
\begin{equation*}
   \sum_{j=0}^{\lfloor k/2 \rfloor} \frac{k^j(k-1)^j}{j\, ! } T^j \leq   \sqrt{\frac{k}{\pi}} \, \exp\left( k + \frac{k}{2} \ln k \right)
\max \left\{ T^{\lfloor k/2 \rfloor}, \, 1 \right\}.
\end{equation*}
Finally, the conclusion follows from Lemma \ref{lem:moment_exist} and \eqref{eq:total_variation}, given that \(u^k_0\) solves \eqref{pb:OC}.
\end{proof}
    
    \begin{rem}
    The argument combines two key ingredients:
\begin{enumerate}
    \item  \emph{Backward moment propagation}: We first employ the quantitative estimates for the evolution of moments backward in time.
    \item \emph{Optimality condition exploitation}: Crucially, we use that $u_0^k$ satisfies the variational condition \eqref{pb:OC}.
\end{enumerate}
The interplay between (i) the analytic moment estimates and (ii) the variational structure of the optimization problem yields the desired error bounds.     \end{rem}

\subsection{Proof of Theorem \ref{thm:rate}}\label{sec_proof_main}
By the definition of the Kantorovich norm, there exists $v^*$, with $\|v^*\|_{\mathcal{C}(\Omega)}\leq 1$ and Lip$(v^*)\leq 1$, such that 
    \begin{equation}\label{eq:diff}
         \| u_0^{k}- u_0^*\|_{\mathrm{Kant}} = \int_{\Omega} v^*\, d (u_0^{k}- u_0^*).  
    \end{equation}
We now estimate the right-hand side expression in several steps.

\medskip
\noindent\textbf{Step 1} (Decomposition of the r.h.s. of \eqref{eq:diff}). By Corollary \ref{cor:approx_poly}, there exists a polynomial $p_k$  of degree $k$ such that
    \begin{equation}\label{eq:p_k}
        \|v^* - p_k\|_{\mathcal{C}(\Omega)} \leq \frac{C_d \, R}{2 k},
 \end{equation}
where \( C_d \) is a constant depending only on the dimension \( d \). 

Let us write
\begin{equation}\label{eq:decomp}
    \int_{\Omega} v^*\, d (u_0^{k}- u_0^*) = \underbrace{\int_{\Omega} (v^* - p_k)\, d (u_0^{k}- u_0^*)}_{\gamma_1} + \underbrace{\int_{\Omega}  p_k\, d (u_0^{k}- u_0^*)}_{\gamma_2}.
\end{equation}
In the following two steps, we estimate upper bounds of $\gamma_1$ and $\gamma_2$, respectively.

\medskip
\noindent\textbf{Step 2} (Estimate on $\gamma_1$). Recalling the definition of the observation error $\epsilon$, by Lemma \ref{prop:a priori bound}, there exists $\tilde{u}_0 \in \mathcal{M}(\Omega)$, such that
\begin{equation}\label{eq:tu}
\displaystyle
\int_{\Omega} x^{\alpha} \, \mathrm{d}\tilde{u}_0(x) 
= \left(e^{-TA}\, \epsilon \right)_{\alpha}, 
\quad \mathrm{for}\ \|\alpha\|_1 \leq k,
\end{equation}
and 
\begin{equation}\label{eq:tu_C}
\displaystyle
\|\tilde{u}_0\|_{\mathrm{TV}} \leq  
\| \epsilon \|_{\infty} \, 
\underbrace{
\sqrt{\frac{k}{\pi}} \, \exp\left( k\left(1+ \frac{2d}{R} + \ln \sqrt{k} \right)\right)
}_{\triangleq \tilde{C}_{d,r}(k)} 
\max \left\{ T^{\lfloor k/2 \rfloor}, \, 1 \right\}.
\end{equation}
Since $u_0^*$ is the true initial distribution, we deduce from \eqref{eq:tu} that 
\begin{equation*}
    \int_{\Omega} x^{\alpha}\, d(\tilde{u}_0+u^*_0)(x) = \left(e^{-TA}\, \mathbf{y} \right)_{\alpha} \quad \mathrm{for }\, \|\alpha\|_1 \leq k.
\end{equation*}
Since $u_0^k$ is a solution of \eqref{pb:OC}, we have
\begin{equation*}
    \|u_0^k\|_{\mathrm{TV}} \leq \|\tilde{u}_0 + u_0^*\|_{\mathrm{TV}}.
\end{equation*}
By the triangle inequality, 
\begin{equation*}
\begin{array}{rl}
  \|u_0^k - u_0^*\|_{\mathrm{TV}} 
  &\leq \|u_0^*\|_{\mathrm{TV}} + \|u_0^k\|_{\mathrm{TV}} 
  \leq \|u_0^*\|_{\mathrm{TV}} + \|\tilde{u}_0 + u_0^*\|_{\mathrm{TV}} \\[1.2ex]
  &\leq 2\|u_0^*\|_{\mathrm{TV}} + \|\tilde{u}_0\|_{\mathrm{TV}}.
\end{array}
\end{equation*}
Combining with \eqref{eq:p_k} and the a priori bound of $\|\tilde{u}_0\|_{\mathrm{TV}}$ in \eqref{eq:tu_C}, it follows that
\begin{equation}\label{eq:gamma_1}
    \gamma_1\leq  \frac{ C_d \, R \, \|u_0^*\|_{\mathrm{TV}}}{k} + \frac{C_d\,R}{2k} \,  \tilde{C}_{d,R}(k)\, \max \left\{ T^{\lfloor k/2 \rfloor}, \, 1 \right\} \,\|\epsilon\|_{\infty}.
\end{equation}

\noindent\textbf{Step 3} (Estimate on $\gamma_2$). By \eqref{eq:p_k}, we have  
\begin{equation*}
    \|p_k\|_{\mathcal{C}(\Omega)} \leq 1 + \frac{C_d \, R}{2 k}.
\end{equation*}
Then, according to Lemma \ref{lem:poly_coef}, 
\begin{equation}\label{eq:c_alpha}
    \sum_{\|\alpha\|_1 \leq k} |c_\alpha| \leq e^{\frac{2dk}{R}} \left( 1 + \frac{C_d \, R}{2 k}\right),
\end{equation}
where \( c_\alpha \) are the coefficients of \( p_k \). Letting \(u_0^k - u_0^*\) be the initial distribution of the heat equation, then the solution at time $T$ have moments $\epsilon_{\alpha}$ by the linearity of the heat equation. Therefore, we deduce from Lemmas \ref{lem:growth} and \ref{lem:Stirling} that for all $\|\alpha\|_1\leq k$,
\begin{equation}\label{eq:x^alpha}
    \left| \int_{\R^d} x^{\alpha} \, d (u_0^k - u_0^* )  \right |\leq  \sqrt{\frac{k}{\pi}} \, \exp\left( k + \frac{k}{2} \ln k \right)
\max \left\{ T^{\lfloor k/2 \rfloor}, \, 1 \right\} \, \| \epsilon \|_{\infty}.
\end{equation}
Recalling the definition of $\gamma_2$ from \eqref{eq:decomp}, we have
\begin{equation}\label{eq:gamma_2_poly}
    \gamma_2 = \sum_{\|\alpha\|_1\leq k} c_{\alpha}  \int_{\R^d} x^{\alpha} \, d (u_0^k - u_0^* ) .
\end{equation}
Recalling the definition of $C_{d,r}(k)$ from \eqref{grosseconstante} and combining \eqref{eq:c_alpha}-\eqref{eq:gamma_2_poly}, we obtain
\begin{equation}
    \label{eq:gamma_2}
        \gamma_2  \leq   \left( 1 + \frac{C_d \, R}{2 k}\right)\, \tilde{C}_{d,R}(k)\, \max \left\{ T^{\lfloor k/2 \rfloor}, \, 1 \right\}\, \| \epsilon \|_{\infty}.
\end{equation}
The final estimate \eqref{eq:conv} follows from \eqref{eq:diff}, \eqref{eq:decomp}, \eqref{eq:gamma_1}, and \eqref{eq:gamma_2}.

\subsection{Error estimate in Wasserstein distance}\label{sec_wasserstein}

The Kantorovich norm (see Definition~\ref{def:kant}) is introduced to quantify the distance between two signed measures that may have different total masses. In the case where the total mass is the same, the Wasserstein distance is a well-known metric for measuring the discrepancy between distributions. Indeed, in this case, we can also estimate the Wasserstein distance between the recovered distribution $u_0^k$ and the true one $u^*_0$, see \eqref{eq:conv2}.

Let $\Omega$ be a compact set. Let \( u_1 \) and \( u_2 \) be two signed measures in \( \mathcal{M}(\Omega) \) with the same total mass, i.e., \( u_1(\Omega) = u_2(\Omega) \). The Wasserstein-1 distance (also referred to as the Kantorovich-Rubinstein distance or the Earth Mover's Distance in the literature) is defined as 
\begin{equation}\label{eq:W1}
    W_1(u_1,u_2) \triangleq  \max_{\begin{array}{c}
    v \in \mathcal{C}(\Omega),\\ 
    \mathrm{Lip}(v) \leq 1
  \end{array}} \int_{\Omega} v \,  d\,(u_1-u_2).
\end{equation}  
Compared to the Kantorovich distance defined in~\eqref{eq:kant_norm}, the \(W_1\) distance removes the upper bound of 1 on the maximum norm of the test function \(v\). As a result, the \(W_1\) distance is not applicable when \(u_1(\Omega) \neq u_2(\Omega)\). Indeed,  adding a constant \(c\) to \(v\) does not change its Lipschitz constant, but introduces an additional nonzero term \(\int_{\Omega} c \, d(u_1 - u_2)\) in the integral. By choosing \(c\) with arbitrarily large magnitude, one can make the right-hand side of~\eqref{eq:W1} diverge to infinity.
 
Restricting \( u_1 \) and \( u_2 \) to the space of probability measures recovers the classical definition of the Wasserstein-1 distance in \cite[Sec.~6]{villani2009}.

In the case of equal total mass, the Wasserstein-1 distance is equivalent to the Kantorovich norm, as shown in~\cite[Eq.\ (1.20)]{hanin1999kantorovich}: If \( u_1(\Omega) = u_2(\Omega) \), then  
\begin{equation}
    \| u_1 - u_2 \|_{\mathrm{Kant}} \leq W_1(u_1,u_2) \leq \max\left\{1\,,\, D_{\Omega}/2\right\} \, \| u_1 - u_2 \|_{\mathrm{Kant}},
\end{equation}  
where \( D_{\Omega} \) denotes the diameter of \( \Omega \).  

Therefore, in the context of Theorem \ref{thm:rate}, if we suppose that there is no error in the total mass observation, i.e., \( \epsilon_0 = 0 \), then,
\begin{equation}\label{eq:conv2}
    W_1(u_0^k, u_0^*) \leq \max\{1\,,\, \sqrt{d} R\} \|u_0^k-u_0^*\|_{\mathrm{Kant}},
\end{equation}
and $ \|u_0^k-u_0^*\|_{\mathrm{Kant}}$ is bounded by \eqref{eq:conv}.

We note that if the total mass of the initial distribution is known a priori (for example, if \( u_0^* \) is known to be a probability measure), then it can be preserved throughout the entire numerical resolution process. First, since the heat equation conserves total mass, we can directly set \( y_0 \) to the known value, without needing to estimate it via quadrature methods. Next, the semigroup \( e^{-tA} \), associated with the reverse moment dynamics, is a polynomial operator. As a result, computing \( e^{-tA} y \) introduces no numerical error, as it simply involves evaluating a polynomial in \( tA \) applied to \( y \); see Section~\ref{sec_discretzation} for the explicit formula. Finally, the moment constraints in the primal problem are given as equalities, ensuring that the total mass is exactly preserved in the final solution.

\section{Observation, discretization, and algorithms}\label{sec_dis}
This section is devoted to the numerical implementation of our moment-based method. Section~\ref{sec:quadrature} introduces two quadrature techniques for observing the moments of \( u(\cdot, T) \), while Section~\ref{sec_discretzation} describes the discretization procedure and the optimization algorithm used to solve \eqref{pb:OC}.

\subsection{Observation of moments}\label{sec:quadrature} 
The numerical calculation of the moments
\begin{equation*}
   M_{\alpha}(T) = \int_{\R^d} x^{\alpha} u(x,T) \,dx.
\end{equation*}
 is based on the observations of \( u(x, T) \) at specific sensor locations. Therefore, quadrature methods are used to approximate the integral. In the following paragraphs, two quadrature methods are presented.

\subsubsection{Uniform Quadrature}
The first and most direct way to compute \( M^{\alpha}(T) \) is via uniform quadrature, that is, by approximating the integral with a Riemann sum on a uniform grid. Specifically, we choose a large \( L>0 \) to bound the domain and discretize \([-L,L]^d\) uniformly using \( n^d \) points (with \( n \) points along each edge). The grid points are denoted by \( x_{\beta} \) for \(\beta\in \mathbb{Z}^{*\,d}_{+}\) with \( \|\beta\|_{\infty} \le n \). The resulting approximation is given by:
\begin{equation}\label{eq:uniform}
    M_{\alpha}(T) \approx \left(\frac{2L}{n}\right)^d \sum_{\|\beta\|_{\infty} \le n} x_{\beta}^{\alpha}\, u(x_{\beta}, T).
\end{equation}
In this method, the sensors are fixed at points \( x_{\beta} \) and hence their positions are independent of \( T \). However, when \( T \) is large, the solution \( u(\cdot, T) \) becomes widely spread, meaning that the concentration within the hypercube \([-L,L]^d\) decreases. 

For example, if the initial distribution is \( \delta_0 \), then by the Gaussian tail estimate,  \( (1 - \epsilon) \) of the mass of \( u(\cdot, T) \) is concentrated in the box \( [-L_{\epsilon, T}, L_{\epsilon, T}]^d \), where \( L_{\epsilon, T} \) is of order \( \sqrt{T \ln(d/\epsilon)} \). 
As a result, the discretization parameter \( n \) must be sufficiently large to accurately capture the solution. In conclusion, the uniform approach is straightforward to implement and works for small terminal times \( T \), but for larger \( T \), we instead propose the Gauss--Hermite method described below.

\subsubsection{Gauss--Hermite quadrature} 
To derive the quadrature formula of this method, we first apply a change of variable to the moment integral by scaling \( x \) by a factor \(\sigma\) (to be determined later):
\begin{equation}\label{eq:Gauss-1}
    \int_{\R^d} x^{\alpha} u(x,T) \,dx = \sigma^{\|\alpha\|_1+d} \int_{\R^d} z^{\alpha} u(\sigma z,T) \,dz.
\end{equation}
Next, we rewrite the integral on the right-hand side of \eqref{eq:Gauss-1} as
\begin{equation}\label{eq:Gauss-2}
    \int_{\R^d} z^{\alpha} u(\sigma z,T) \,dz = \int_{\R^d} \underbrace{\Bigl( z^{\alpha} u(\sigma z,T) e^{\|z\|^2} \Bigr)}_{\triangleq \,g_{\sigma}(z)}\, e^{-\|z\|^2} \,dz.
\end{equation}
By applying the Gauss--Hermite quadrature in dimension \(d\), we obtain that
\begin{equation}\label{eq:Gauss-3}
    \int_{\R^d} g_{\sigma}(z) \, e^{-\|z\|^2} \,dz \approx \sum_{i_1=1}^n\cdots\sum_{i_d=1}^n \Biggl(\prod_{j=1}^d \omega_{i_j}\Biggr) \, g_{\sigma}(z_{i_1}, \ldots, z_{i_d}),
\end{equation}
where $n\geq 1$ is the degree of quadrature, \((\omega_i)_{i=1}^n\) and \((z_i)_{i=1}^n\) denote the weights and nodes corresponding to the \(n\)-th Hermite polynomial (see \cite[Sec.~18.3]{olver2010nist}).

For the quadrature to succeed, it is essential that \(g_{\sigma}(z_{i_1}, \ldots, z_{i_d})\) remains within moderate bounds. Otherwise, numerical overflow may occur, especially since \(g_{\sigma}\) contains the exponential factor \(e^{\|z\|^2}\). However, the rapid spatial decay of the heat solution \(u(\cdot,T)\) compensates for this exponential growth if a reasonable scaling factor \(\sigma\) is chosen.
Indeed, note that \(u(\sigma z, T)\) behaves like \(e^{-\|\sigma z\|^2/(4T)}\) (according to the formula for the heat kernel \(G\)). Hence, we choose
\begin{equation}\label{eq:Gauss-4}
    \sigma = 2\sqrt{T}.
\end{equation}
Therefore, by \eqref{eq:Gauss-1}--\eqref{eq:Gauss-4} we obtain the following complete approximation for \(M_{\alpha}(T)\):
\begin{equation}
M_{\alpha}(T) 
\approx (2\sqrt{T})^{\|\alpha\|_1 + d} \sum_{i_1\ldots, i_d = 1}^n  
\left( \prod_{j = 1}^d \omega_{i_j} \, z_{i_j}^{\alpha_j} \right) \nonumber\exp\left( \sum_{j = 1}^d z_{i_j}^2 \right) 
\, u\left( 2\sqrt{T}(z_{i_1}, \ldots, z_{i_d}), T \right)\!.
\label{eq:Gauss}
\end{equation}
Compared to \eqref{eq:uniform}, the Gauss--Hermite quadrature features sensor positions that depend on \( T \). Indeed, the positions 
\[
2\sqrt{T}(z_{i_1},\ldots, z_{i_d})
\]
grow at a rate proportional to \(\sqrt{T}\). Despite the inconvenience associated with moving sensors, our numerical simulations indicate that the Gauss--Hermite quadrature \eqref{eq:Gauss} yields superior results to the uniform quadrature \eqref{eq:uniform} with the same number of sensors, see Figure \ref{fig:quadrature} for a comparison.

In the presence of pointwise observational noise, 
the following result provides an upper bound for the moment error produced by 
the Gauss--Hermite quadrature \eqref{eq:Gauss}, expressed in terms of the number 
of basis functions $n$ (per coordinate direction) and the noise level $\eta$.

\begin{lem}\label{lm:quadrature}
Let $T>0$ and let $\alpha\in\mathbb{N}^d$ be a multi–index.  
Assume that in the Gauss--Hermite quadrature formula \eqref{eq:Gauss}, each value
\(
u\!\left( 2\sqrt{T}\,z_{i},\,T\right)
\)
is observed with an additive perturbation of magnitude at most $\eta>0$.  
Let $M_\alpha^n(T)$ denote the quadrature output computed from these noisy evaluations, using $n$ Hermite nodes in each coordinate direction.
Then, there exists constants
\(\rho\in (0,1),
C_{\alpha,d,T}>0
\) and $C_{\alpha,d}$
such that the quadrature error satisfies
\[
\bigl| M_\alpha(T) - M_\alpha^n(T) \bigr|
\;\le\;
C_{\alpha,d,T}\,
\rho^{n}
\;+\;C_{\alpha,d}\,
\eta,
\qquad 
\forall\, n\ge (\|\alpha\|_1+1)/2.
\]
\end{lem}

\begin{proof}
The proof is presented in Section~\ref{sec_proof}.
\end{proof}

By combining the quadrature error bound derived above with the error estimate established in Theorem~\ref{thm:rate}, we obtain the following result.

\begin{rem}[Total error induced by pointwise observational noise]\label{rem:total-error}
In the setting of Theorem~\ref{thm:rate}, suppose that the moments are computed
using the Gauss--Hermite quadrature formula \eqref{eq:Gauss} with 
$n$ nodes in each coordinate direction, and assume that each quadrature
evaluation is perturbed by at most $\eta>0$.  
By Lemma~\ref{lm:quadrature}, for every $n \ge (k+1)/2$, there exist constants $\rho\in (0,1)$ $C_1,C_2,C_3>0$ (independent of $n$ and $\eta$) such that
\begin{equation}\label{eq:conv_3}
\|u_0^k - u_0^*\|_{\mathrm{Kant}}
\;\le\;
\frac{C_1}{k}
\;+\;
C_2\, \rho^n
\;+\;
C_3\, \eta .
\end{equation}
Here \(C_1\) depends only on \(d\), \(R\), and \(\|u_0^*\|_{\mathrm{TV}}\), 
while \(\rho\), \(C_2\) and \(C_3\) may additionally depend on  
\(T\) and \(k\).  
Explicit formulas can be extracted from \eqref{grosseconstante} and the proof of
Lemma~\ref{lm:quadrature}. The total error \eqref{eq:conv_3} decomposes into three contributions: (1) the intrinsic approximation error due to truncating the moment system,
        which decays like $1/k$; (2) the Gauss--Hermite quadrature error, which decays exponentially in $n$; and (3) the pointwise observational noise, contributing linearly in~$\eta$.
\end{rem}

\subsection{Discretize-then-optimize}\label{sec_discretzation}
Since \eqref{pb:OC} is an optimization problem defined on an infinite-dimensional space, we adopt a discretize-then-optimize approach to solve it.
\subsubsection{Discretization} We first discretize the domain $\Omega$ by selecting a set of $N$ points,
\[
\Omega \mathrel{\mathop{\hbox to 2cm{\rightarrowfill}}\limits^{\mathrm{discretize}}} \Omega_h = \{x_1, \ldots, x_N\}.
\]
In this discrete setting, any initial condition \( u_0 \in \mathcal{M}(\Omega_h) \) can be expressed as a finite sum of Dirac measures:
\[
u_0 = \sum_{i=1}^N m_i\, \delta_{x_i}.
\]
Thus, the discretized version of \eqref{pb:OC} reduces to a finite-dimensional optimization problem in \(\mathbb{R}^N\), with the optimization variables given by the vector of weights \(\mathbf{m} = (m_i)_{i=1}^N\):
\begin{equation}\label{pb:OC_dis}
    \inf_{\mathbf{m}\in \mathbb{R}^N} \|\mathbf{m}\|_{\ell^1} 
    \quad
    \mathrm{subject\, to}
    \quad
    B\, \mathbf{m} \;=\; e^{-T A}\, \mathbf{y},
\end{equation}
where $\mathbf{y} = (y_{\alpha})_{\|\alpha\|_1\leq k}$ denotes the observations of moments at time $T$, and
\begin{itemize}
    \item $A$ is defined in \eqref{eq:ODE_moments}. By Lemma~\ref{lem:growth}, $A$ is \( (\lfloor k/2 \rfloor +1) \)-nilpotent, so that
    \[
    e^{-T A} = \sum_{j=0}^{\lfloor k/2 \rfloor} \frac{(-A)^j}{j!}\, T^j ;
    \]
    \item $B$ is a $\binom{k + d}{d} \times N$ matrix with entries
    \[
    b_{\alpha, i} = x_i^{\alpha},\quad \mathrm{for } \, \|\alpha\|_1 \leq k,\quad i=1,\ldots, N.
    \]
\end{itemize}

The non-emptiness of the admissible set for the discretized problem \eqref{pb:OC_dis} is nontrivial. This is because the monomials $\{x^{\alpha}\}_{\|\alpha\|_1\leq k}$ are not, in general, linearly independent in $\mathcal{C}(\Omega_h)$. To guarantee the feasibility of \eqref{pb:OC_dis}, we introduce the notion of a \textit{unisolvent set}.
 A set of points \( \{z_{i} \in \mathbb{R}^d\}_{i=1,\ldots, \binom{k + d}{d} }\) is said to be a unisolvent set of degree \( k \) if and only if the only polynomial \( p \) of degree at most \( k \) that satisfies  
\[
p(z_{i}) = 0 \quad \forall\, i\leq \binom{k + d}{d} 
\]
is the zero polynomial, i.e., \( p \equiv 0 \). 
Here, we give two examples of unisolvent sets:
\begin{enumerate}
    \item If $d=1$, then \( \{z_{i} \in \mathbb{R}\}_{i =1,\ldots, k+1} \) is a unisolvent set of degree \( k \) if and only if $z_{i}\neq z_{j}$ for any $i\neq j$.
    \item If \(d = 2\), then the family of Padua points \cite{caliari2005bivariate,bos2007bivariate} up to degree \(k\) forms a unisolvent set of degree \(k\). Suppose \(k\) is even, then the family of Padua points up to degree \(k\) is the collection of $z_{i,j} = (z^1_i,z^2_j) \in\R^2$ for $0\leq i\leq k$ and $0\leq j\leq k/2$ with
\begin{equation*}
z^1_i = \cos\left(\frac{i\,\pi}{k}\right),
\quad
z^2_j =
\left\{
  \begin{array}{ll}
    \displaystyle \cos\left(\frac{2j\,\pi}{k+1}\right), 
    & \quad \mathrm{if}\ i\ \mathrm{is\ odd}, \\[1.2em]
    \displaystyle \cos\left(\frac{(2j+1)\,\pi}{k+1}\right), 
    & \quad \mathrm{if}\ i\ \mathrm{is\ even}.
  \end{array}
\right.
\end{equation*}

\end{enumerate}

\begin{lem}\label{lem:exist}
   Suppose that \(N \ge \binom{k + d}{d}\) and that \(\Omega_h\) contains \(\binom{k + d}{d}\) points forming a unisolvent set of degree \(k\). Then  problem \eqref{pb:OC_dis} admits solutions for any $\mathbf{y} \in \R^{\binom{k+d}{d}}$.
\end{lem}

\begin{proof}
    The matrix \( B \) is a high-dimensional (rectangular) Vandermonde matrix. It has full row rank if it contains a square, invertible Vandermonde submatrix. According to \cite[Thm.\@ 4.1]{olver2006multivariate}, this condition is equivalent to the existence of \(\binom{k + d}{d}\) points in \(\Omega_h\) that form a unisolvent set of degree \(k\). In this case, the admissible set is non-empty and convex. The existence of a solution then follows from standard arguments based on the convexity and coercivity of problem~\eqref{pb:OC_dis}.
\end{proof}

\subsubsection{Optimization}\label{sec_optimization}
Following the general idea in \cite{chen2001atomic}, we transform \eqref{pb:OC_dis} into an equivalent linear programming and employ the simplex method to solve it. 
The equivalent linear programming for \eqref{pb:OC_dis} is given by
\begin{equation}\label{pb:OC_dis_eq}
    \inf_{\mathbf{m_+}\in \mathbb{R}_+^N, \, \mathbf{m_-}\in \mathbb{R}_+^N} 
    \langle \mathbf{1}, \mathbf{m_+} + \mathbf{m_-} \rangle 
    \quad
    \mathrm{subject\, to}
    \quad
    B \bigl(\mathbf{m_+} - \mathbf{m_-}\bigr) = e^{-T A}\, \mathbf{y} ,
\end{equation}
where \(\mathbf{1}\) is the \(N\)-dimensional vector of all ones.

\begin{prop}\label{prop:simplex}
      Under the setting of Lemma \ref{lem:exist}, problem \eqref{pb:OC_dis_eq} has solutions. Let \((\mathbf{m_+} , \mathbf{m_-})\) be the solution of \eqref{pb:OC_dis_eq} obtained from the simplex method. Then \( \mathbf{m} = \mathbf{m_+} - \mathbf{m_-} \) is a solution of \eqref{pb:OC_dis} and $\|\mathbf{m}\|_{\ell^0} \leq \binom{k + d}{d}$.
\end{prop}
\begin{proof}
    The proof is the same as the proof of Theorem 4.7 in \cite{liu2025representation}.
\end{proof}

\section{Numerical simulations}\label{sec_num}

\subsection{Settings}
In our numerical experiments, we consider the heat equation \eqref{eq:heat} in dimensions \(d=1\) and \(d=2\), with initial data given by six Dirac masses whose locations and amplitudes are chosen at random. The exact positions and amplitudes are listed in Table \ref{tab:dirac} and depicted in the first row of Figure \ref{fig:initial-final}. Here, the amplitudes are randomly generated and kept deliberately heterogeneous to provide a challenging test case and to assess the robustness of our recovery method.

\begin{table}[h]
\caption{\label{tab:dirac}Support positions and amplitudes of $u_0^*$ in 1D and 2D.}
\begin{center}
\begin{tabular}{@{}c@{\hspace{15mm}}rr@{\hspace{15mm}}lr}
\br
\multirow{2}{*}{Index} & \multicolumn{2}{c}{\hspace{-15mm}1D Case} & \multicolumn{2}{c}{\hspace{-5mm}2D Case} \\
 & Position & Amplitude & Position & Amplitude \\
\mr
1 & $-3.11$ &  $4.0071$ & $(-1.30,\ -2.27)$ &  $2.6832$ \\
2 &  $2.16$ & $-4.6658$ & $(-1.43,\ \ \ 0.08)$ &  $0.6610$ \\
3 & $-2.13$ &  $4.5695$ & $(\ \ 3.90,\ -3.69)$ & $-2.5463$ \\
4 &  $0.30$ & $-3.6279$ & $(\ \ 3.72,\ \ \ 2.57)$ &  $0.4501$ \\
5 & $-4.37$ & $-2.1617$ & $(\ \ 3.04,\ -0.91)$ & $-3.5543$ \\
6 &  $3.77$ &  $1.0608$ & $(-3.96,\ -0.52)$ & $-0.5107$ \\
\br
\end{tabular}
\end{center}
\end{table}

The moments are computed using quadrature methods based on the values of \( u(x_i, T) \), where \( x_i \) denotes the positions of the sensors used in the corresponding quadrature rule. The values \( u(x_i, T) \) are computed using the closed-form solution of the heat equation. Since the initial distribution is a sum of Dirac delta functions, the solution is expressed as a sum of Gaussian distributions. This approach avoids numerical errors associated with discretizing the heat equation.
Aside from this, the initial distribution is not incorporated into the optimization process.
The second row of Figure \ref{fig:initial-final} shows solutions at one terminal time \( T=10 \) over their proper domain, clearly demonstrating that recovering the initial distribution \( u_0 \) is highly non-trivial.

\begin{figure}[h]
  \centering
  \begin{subfigure}[b]{0.45\textwidth}
    \centering
    \includegraphics[width=\linewidth]{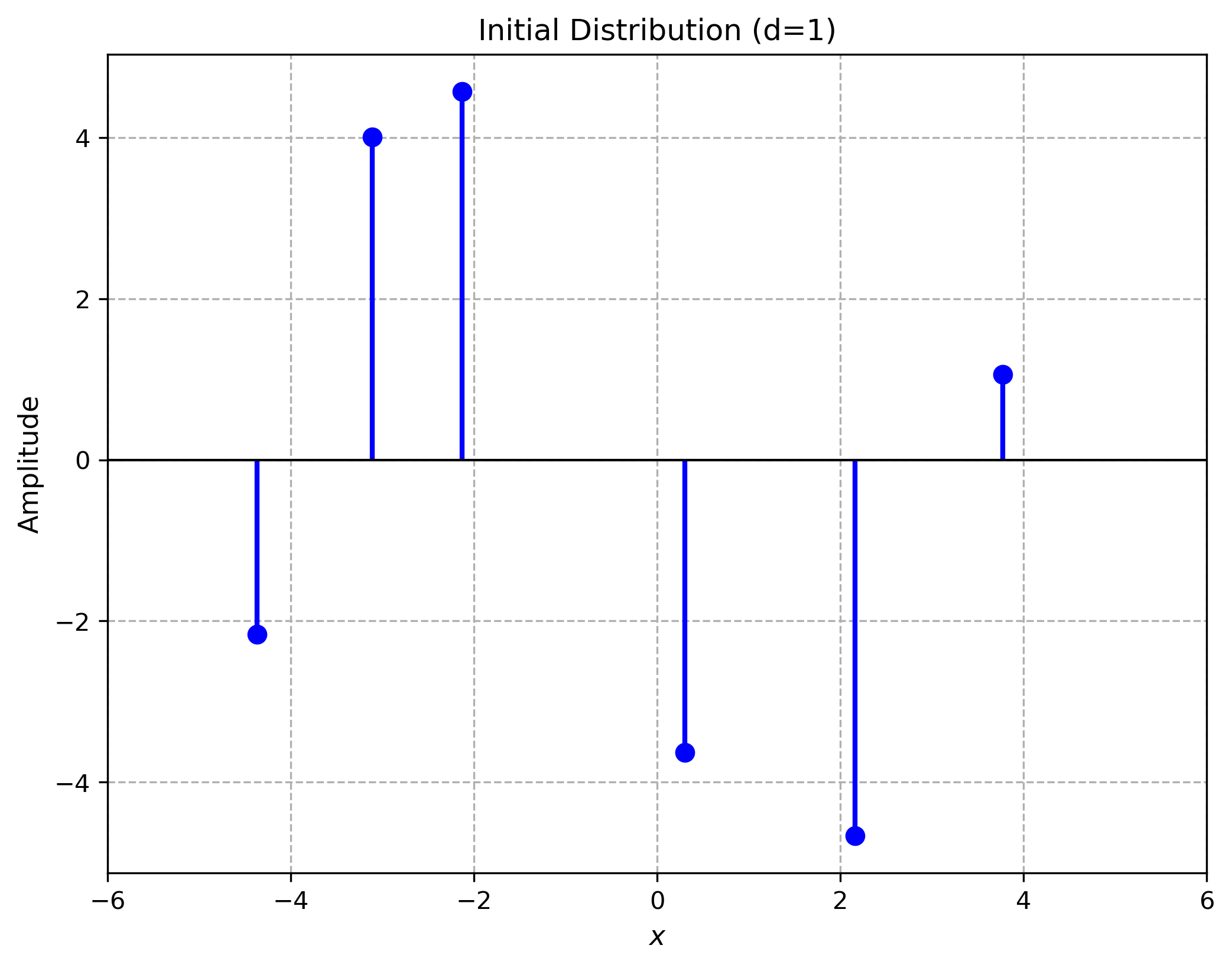}
    \caption{Initial distribution in 1D.}
    \label{fig:init-1d}
  \end{subfigure}
  \hfill
  \begin{subfigure}[b]{0.45\textwidth}
    \centering
    \includegraphics[width=\linewidth]{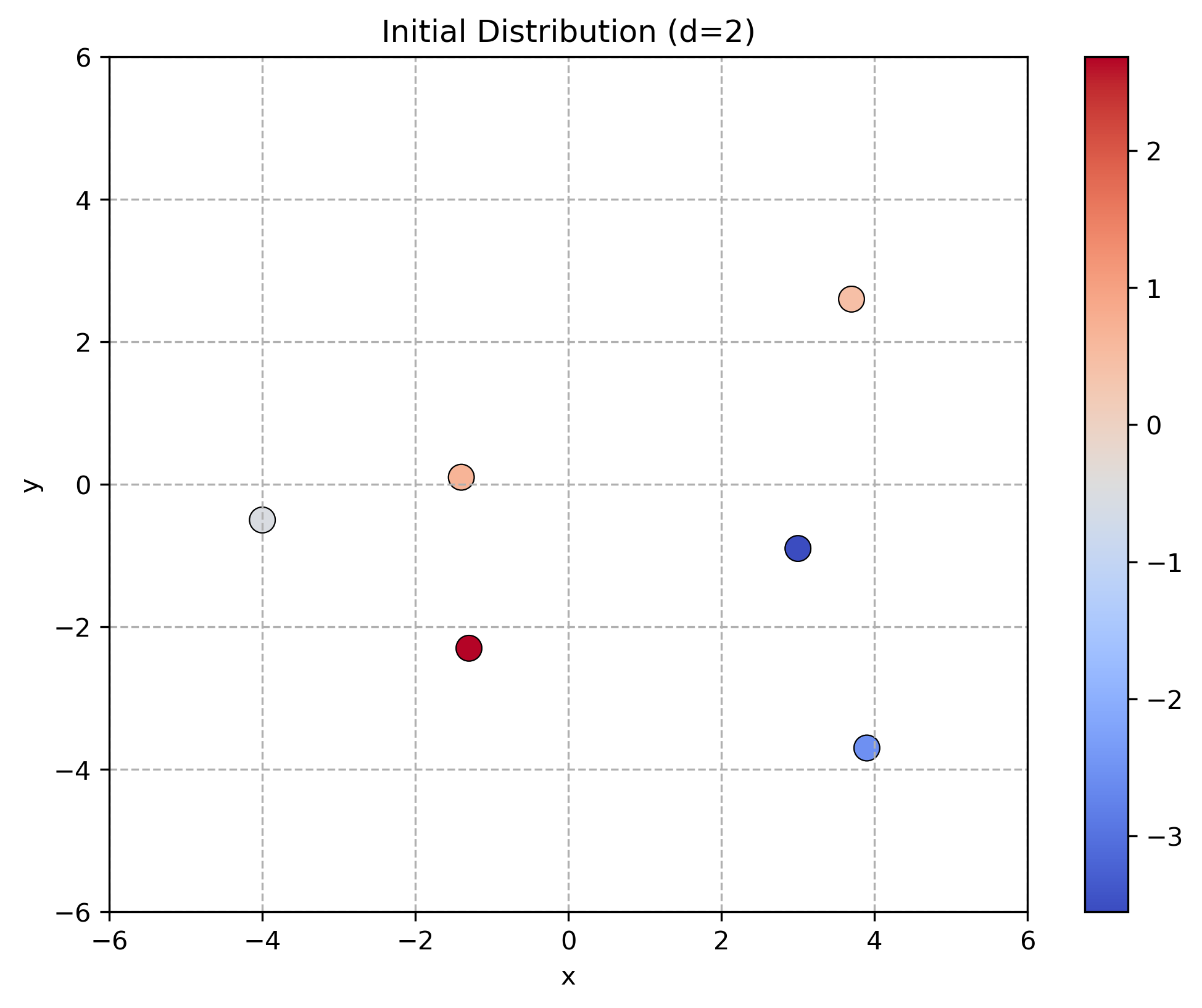}
    \caption{Initial distribution in 2D.}
    \label{fig:term-1d}
  \end{subfigure}

  \vspace{1em}  

  \begin{subfigure}[b]{0.45\textwidth}
    \centering
    \includegraphics[width=\linewidth]{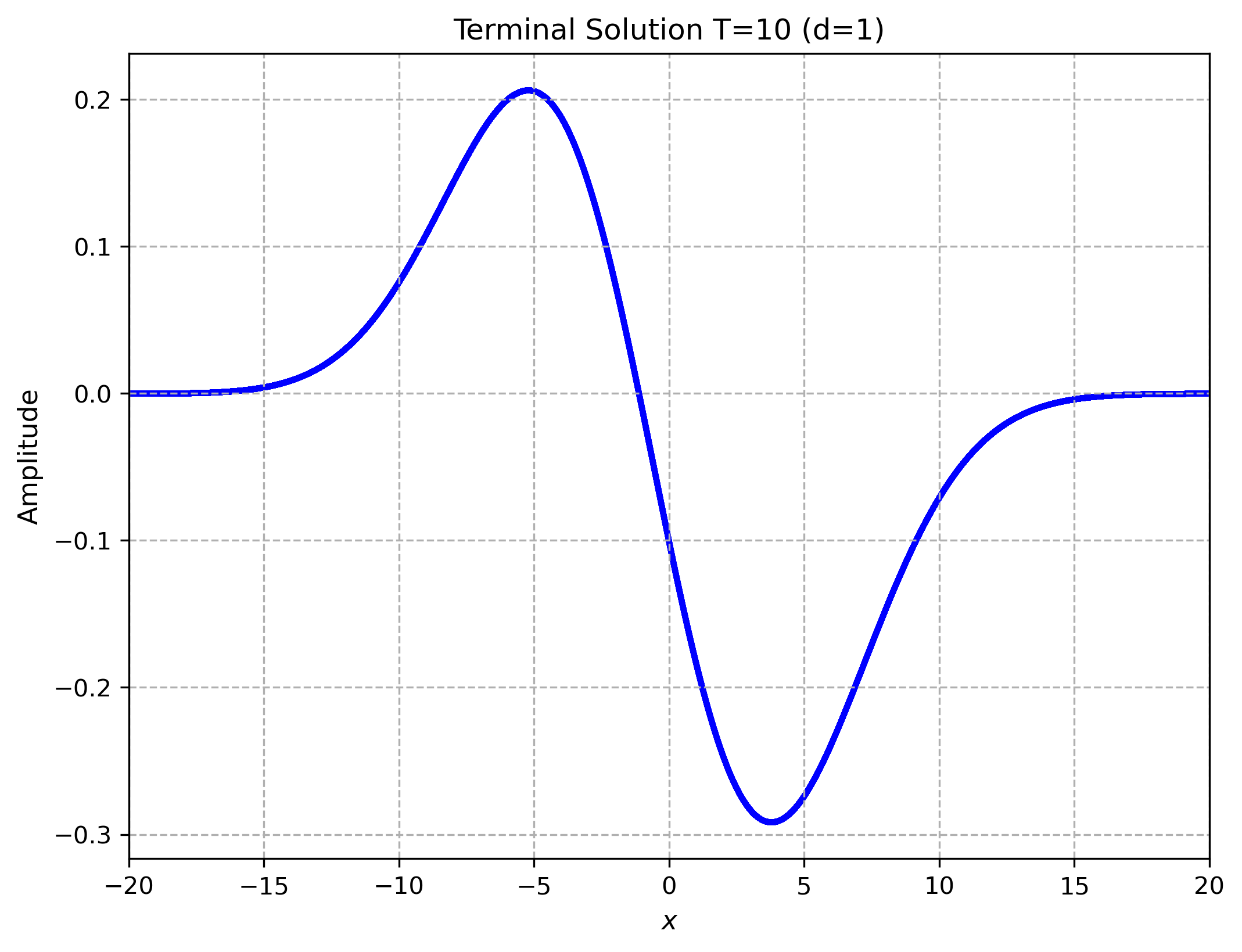}
    \caption{Terminal solution in 1D.}
    \label{fig:init-2d}
  \end{subfigure}
  \hfill
  \begin{subfigure}[b]{0.45\textwidth}
    \centering
    \includegraphics[width=\linewidth]{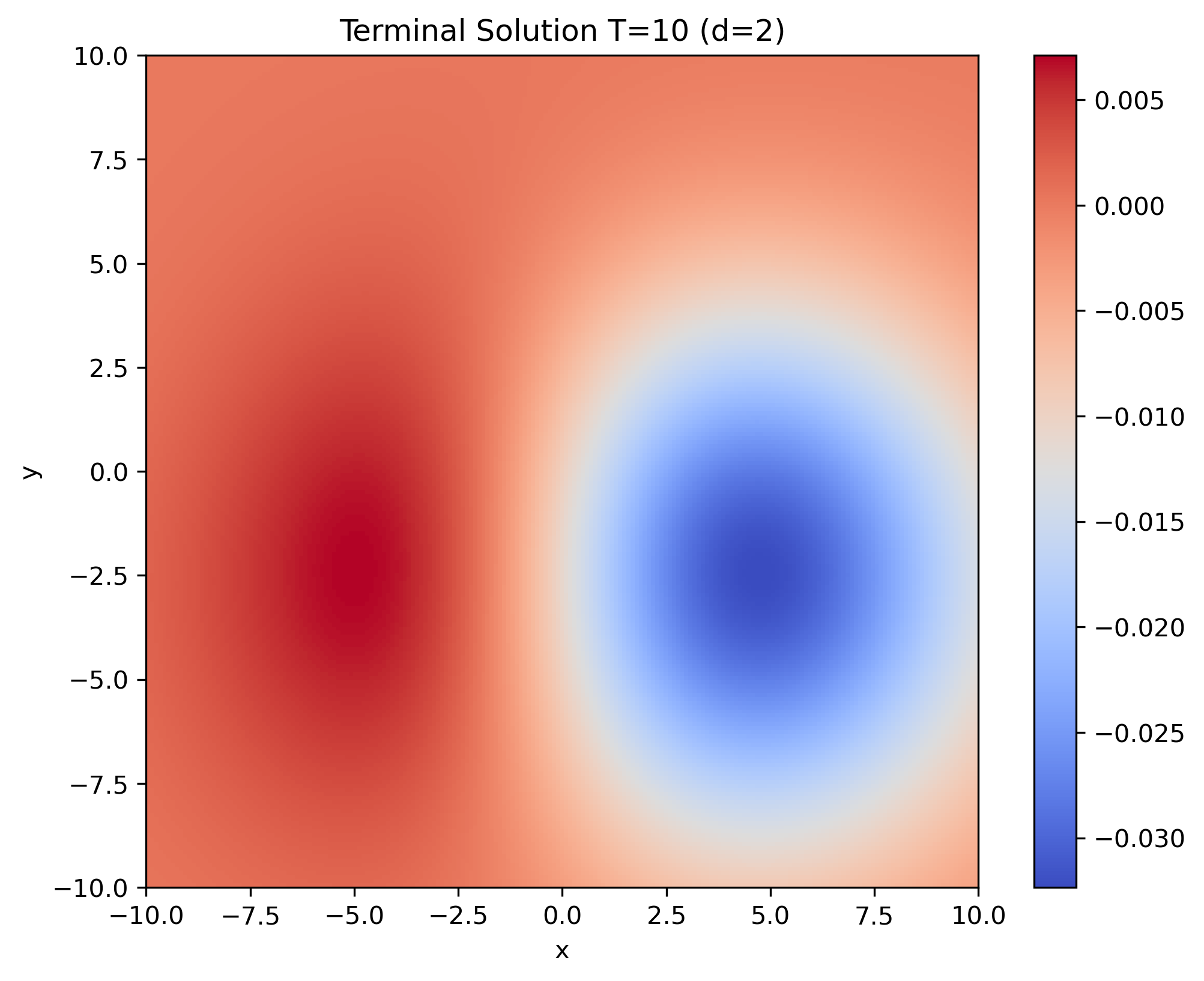}
    \caption{Terminal solution in 2D.}
    \label{fig:term-2d}
  \end{subfigure}

  \caption{Initial conditions and terminal solutions ($T=10$) of \eqref{eq:heat}.}
  \label{fig:initial-final}
\end{figure}

\subsubsection{Quadrature of moments}
 We present here the settings and observations of quadrature methods to obtain the moments in the one-dimensional case; the behavior in two dimensions is qualitatively similar.

For the uniform quadrature method, we integrate over the domain \([-50, 50]\) and vary the number of discretization points \( n \) from 2 to 100. For the Gauss--Hermite quadrature, we use up to the first 100 weights and nodes associated with the Hermite polynomials. Figure~\ref{fig:quadrature}(a) compares the error of these two methods for the fourth-order moment of \( u(\cdot, T) \), with \( T = 10 \) fixed. The results clearly show that Gauss--Hermite quadrature achieves greater accuracy than the uniform method for the same number of sensors. However, when the number of sensors is sufficiently large, both quadrature schemes yield reliable moment estimates at \( T = 10 \). This implies that for moderately sized terminal times, a sufficiently dense uniform grid can replace Gauss--Hermite quadrature without the need to move sensor locations according to the Gaussian weights in \eqref{eq:Gauss}. 

For larger values of \( T \), as shown in Figure~\ref{fig:quadrature}(b), the errors from both quadrature methods increase, but Gauss--Hermite quadrature remains significantly more stable than the uniform approach. In fact, for large \( T \), the uniform method requires both a larger integration domain and a substantially greater number of sensors to achieve comparable accuracy.

In Figure~\ref{fig:quadrature}(c), we compare the moment errors at time \(0\) obtained via Gauss--Hermite quadrature combined with inversion of the moment equation~\eqref{eq:ODE_moments}. The error is defined by
\begin{equation}\label{eq:error}
    \|M(0) - e^{-TA} \, \mathbf{y}\|_{\infty},
\end{equation}
where \(A\) is the generator of the semigroup in~\eqref{eq:ODE_moments}, and \(\mathbf{y} = (y_0, \ldots, y_k)\) is computed using Gauss--Hermite quadrature with 100 sensors.

\begin{figure}[h]
  \centering
  \begin{subfigure}[b]{0.45\textwidth}
    \centering
    \includegraphics[width=\linewidth]{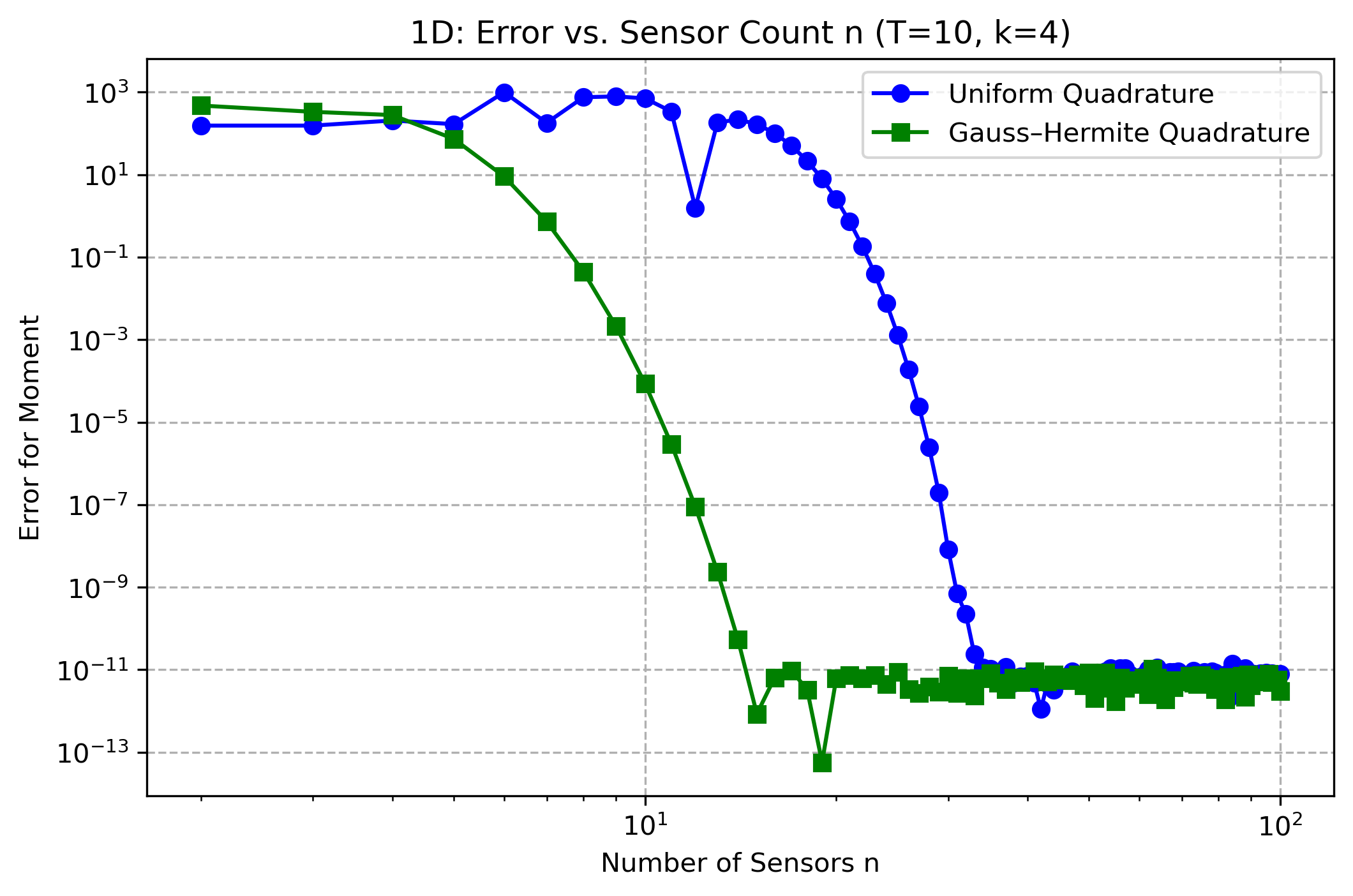}
    \caption{Error of the fourth moment as a function of the sensor count \(n\) for the fixed terminal time \(T = 10\).}
    \label{fig:moment-1d}
  \end{subfigure}
  \hfill
  \begin{subfigure}[b]{0.45\textwidth}
    \centering
    \includegraphics[width=\linewidth]{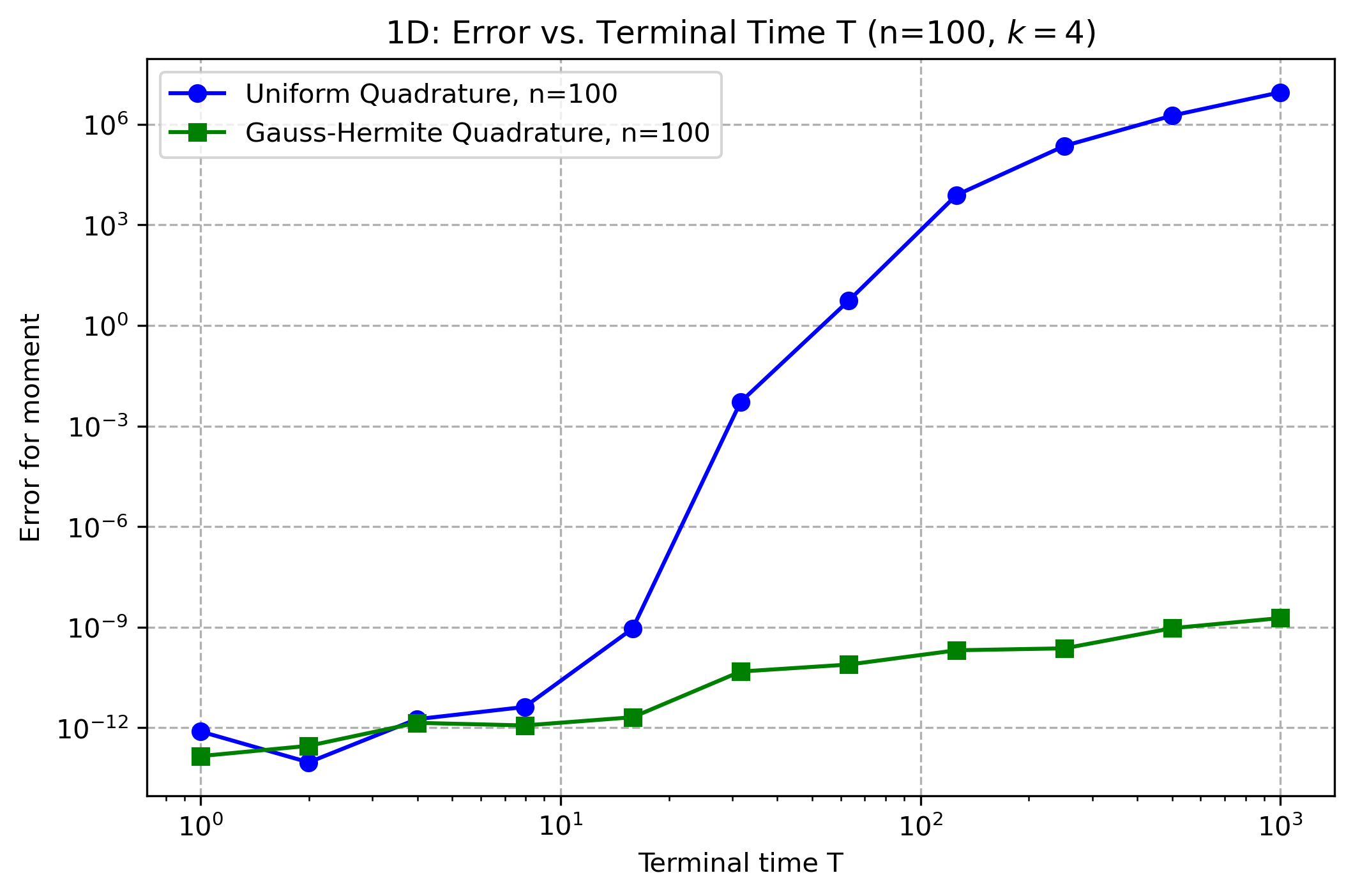}
    \caption{Error of the fourth moment as a function of the terminal time \(T\) with a fixed number of sensors \(n = 100\).}
    \label{fig:moment-t-1d}
  \end{subfigure}

  \vspace{2em}

  \begin{subfigure}[b]{0.45\textwidth}
    \centering
    \includegraphics[width=\linewidth]{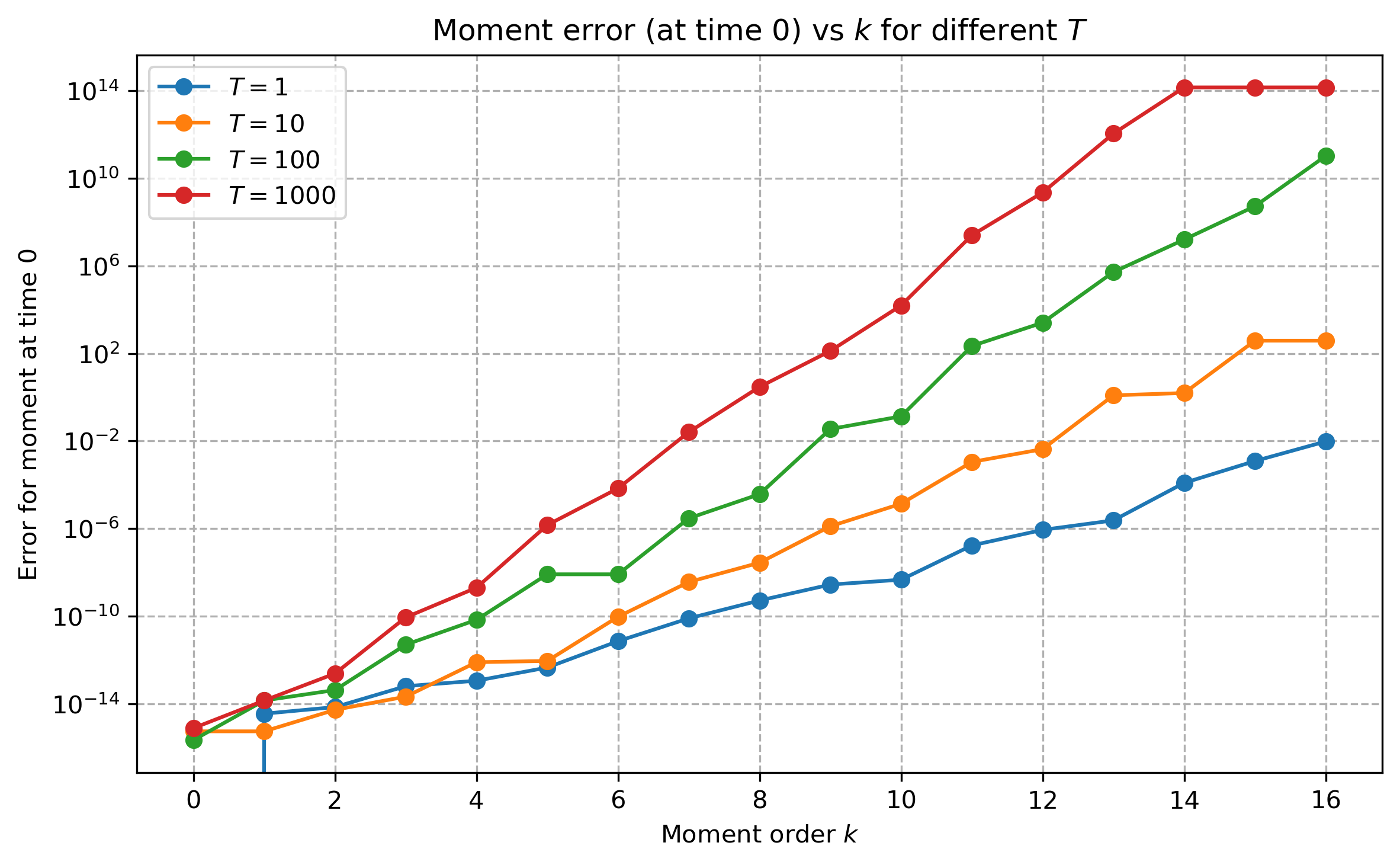}
    \caption{Error of the \(k\)-th moment at time 0 (see \eqref{eq:error}) for \(k = 0,\dots,16\) with terminal times \(T = 1, 10, 100, 1000\), using the Gauss--Hermite method with \(n = 100\) sensors.}
    \label{fig:moment-0-1d}
  \end{subfigure}

  \caption{Moment errors obtained with the quadrature methods.  
           (a) Varying sensor count \(n \in \{1,\dots,100\}\) at fixed \(T = 10\).  
           (b) Varying terminal time \(T \in [1, 1000]\) at fixed \(n = 100\).  
           (c) Joint influence of moment order \(k\) and terminal time for the Gauss--Hermite scheme.}
  \label{fig:quadrature}
\end{figure}

\subsubsection{Discretization and optimization parameters}
Based on our previous quadrature results, we use the values \(y_{k}\) obtained via Gauss--Hermite quadrature with \(10^2\) and \(10^4\) sensors in one and two dimensions, respectively.  The a priori domain \(\Omega\) for the initial distribution is taken to be \([-5,5]\) in one dimension and \([-5,5]^2\) in two dimensions.  

We then discretize \(\Omega\) using a uniform mesh, yielding the discrete domain \(\Omega_h\) and the fully discretized problem \eqref{pb:OC_dis}.  In the one-dimensional case, this mesh is a unisolvent set of degrees of discretization.  In two dimensions, the uniform mesh achieves results comparable to those obtained with the Padua points, so we opt for the simpler uniform grid. The matrix \( B \) in the discretized problem \eqref{pb:OC_dis} has dimensions \( N \times \binom{k+d}{d} \). This matrix must be stored in the computer's RAM when applying the simplex method. To accommodate this constraint, we set \( N = 10^3 \) for the one-dimensional case and \( N = 10^4 \) for the two-dimensional case. These choices limit the maximum allowable moment order \( k \), as larger values can lead to memory overflow. All numerical simulations are performed on a laptop equipped with 16 GB of RAM. To avoid any loss of numerical precision in the matrix \( B \), we fix the maximum moment order to \( k_{\mathrm{max}} = 16 \) in one dimension and \( k_{\mathrm{max}} = 10 \) in two dimensions. For the 1D case, we test with terminal times \( T = 1, 10, 100, 1000 \), and for the 2D case, $T$ is set to $100$. The table below summarizes all the parameters used in our experiments.

\begin{table}
\caption{\label{tab:experiment-parameters}Summary of parameters used in the numerical experiments.}
\begin{tabular*}{\textwidth}{@{}lll}
\br
Parameter & 1D Case & 2D Case \\
\mr
Quadrature method & Gauss--Hermite & Gauss--Hermite \\
Number of sensors $n$ & \(10^2\) & \(10^4\) \\
A priori domain \(\Omega\) & \([-5, 5]\) & \([-5, 5]^2\) \\
Discretization method & Uniform mesh (unisolvent) & Uniform mesh \\
Discretization size \(N\) & \(10^3\) & \(10^4\) \\
Max moment order \(k_{\mathrm{max}}\) & 16 & 10 \\
Terminal times \(T\) & 1, 10, 100, 1000 & 100 \\
\mr
Hardware & \multicolumn{2}{@{}l}{\;\, Intel(R) Core(TM) i5-1335U, 16 GB RAM} \\
\br
\end{tabular*}
\end{table}

\subsection{Numerical results}

\subsubsection{The one-dimensional case} 
We first vary the final time $T \in \{1,10,100,1000\}$, the level of pointwise
observational noise with standard deviations $10^{-32}$, $10^{-16}$,
$10^{-8}$, and $10^{-4}$, and the moment order $k\in\{0,\dots,16\}$.  
For each pair $(T,\mathrm{noise})$, we solve the discretized optimization problem and compute the Wasserstein distance $W_1$ between the reconstructed and the true initial distributions. To better assess the robustness of our method, independently of the overall amplitude scale, we report the normalized error $W_1/\|u_0^\ast\|_{\mathrm{TV}}$, where $\|u_0^\ast\|_{\mathrm{TV}}$ is the total variation of the true initial distribution.

The optimal moment order $k^*$ and the corresponding normalized error $W_1^*/\|u_0^\ast\|_{\mathrm{TV}}$ for each time--noise pair are reported in
Table~\ref{tab:noise-vs-T}.  
A clear trend emerges: the optimal moment order $k^*$ decreases as either
the final time $T$ increases or the observation noise level becomes larger,
and the normalized error $W_1^*/\|u_0^\ast\|_{\mathrm{TV}}$ deteriorates accordingly. This behavior is consistent with our estimate \eqref{eq:conv} and the 
discussion in Remark~\ref{rem:moment_opt}. 

For large final times such as $T = 1000$, even a very small pointwise noise 
(e.g., with standard deviation $10^{-8}$) leads to a relatively large reconstruction error.
In practice, to achieve an accurate reconstruction in this regime, it is therefore necessary that the observational noise be much smaller than the amplitude of the heat solution at time $T$.

Table~\ref{tab:dirac-comparison} provides a more detailed view in the noise-free setting (up to machine precision).
For each $T$, we select the optimal moment order $k^*$ from $k=0,\dots,16$ and list the six Dirac components
with largest amplitudes, comparing their recovered positions and amplitudes with the exact ones.
For $T=1$ and $T=10$, the agreement is essentially exact, while for $T=100$ and $T=1000$ the amplitudes deteriorate
but the recovered support points remain close to the true locations.
This confirms that, even in the large-time regime where $W_1^*/\|u_0^\ast\|_{\mathrm{TV}}$ remains of order $10^0$,
the method still captures meaningful information on the underlying support, which could be further exploited
by a sparse post-processing or local refinement strategy.

\begin{table}[h!]
\centering
\footnotesize
\caption{Numerical results for the 1D case: Optimal moment order $k^*$ and corresponding normalized Wasserstein 
distance $W_1^*/\|u_0^\ast\|_{\mathrm{TV}}$ for each noise level and final time $T$.}
\label{tab:noise-vs-T}
\setlength{\tabcolsep}{9pt}
\renewcommand{\arraystretch}{1.2}

\begin{tabular}{c ccccccccc}
\toprule
\multirow{2}{*}{Noise std} 
 & \multicolumn{2}{c}{$T = 1$}
 & \multicolumn{2}{c}{$T = 10$}
 & \multicolumn{2}{c}{$T = 100$}
 & \multicolumn{2}{c}{$T = 1000$} \\
\cmidrule(lr){2-3}\cmidrule(lr){4-5}\cmidrule(lr){6-7}\cmidrule(lr){8-9}
 & $k^*$ & $W_1^*/\|u_0^\ast\|_{\mathrm{TV}}$ 
 & $k^*$ & $W_1^*/\|u_0^\ast\|_{\mathrm{TV}}$ 
 & $k^*$ & $W_1^*/\|u_0^\ast\|_{\mathrm{TV}}$
 & $k^*$ & $W_1^*/\|u_0^\ast\|_{\mathrm{TV}}$ \\
\midrule

$10^{-32}$ 
 & 16 & $9.16\times10^{-9}$
 & 14 & $2.96\times10^{-7}$
 & 12 & $1.18\times10^{-1}$
 & 10 & $2.87\times10^{-1}$ \\

$10^{-16}$ 
 & 14 & $3.00\times10^{-3}$
 & 10 & $3.18\times10^{-1}$
 & 7  & $5.57\times10^{-1}$
 & 5  & $7.12\times10^{-1}$ \\

$10^{-8}$  
 & 8  & $4.46\times10^{-1}$
 & 5  & $6.97\times10^{-1}$
 & 3  & $8.21\times10^{-1}$
 & 3  & $9.85\times10^{-1}$ \\

$10^{-4}$  
 & 4  & $7.32\times10^{-1}$
 & 2  & $1.22\times10^{0}$
 & 1  & $1.40\times10^{0}$
 & 0  & $1.43\times10^{0}$ \\
\bottomrule
\end{tabular}
\end{table}

\begin{table}[h!]
\centering
\footnotesize
\caption{Numerical results for the 1D case: Exact and recovered Dirac Positions (Pos) and amplitudes (Amp) for different final times $T$. In this simulation, the pointwise observation error is zero (up to machine precision). For each $T$, the optimal moment order $k^*$ is selected from $k=0,\dots,16$. The Dirac components are sorted in decreasing order of amplitude, and the table reports the six largest ones (matching the support size of the exact distribution). All numerical values are rounded to two decimal digits.}
\label{tab:dirac-comparison}
\setlength{\tabcolsep}{9pt}
\renewcommand{\arraystretch}{1.2}

\begin{tabular}{c ccccccccccc}
\toprule
\multirow{2}{*}{Index}
  & \multicolumn{2}{c}{Exact}
  & \multicolumn{2}{c}{$T = 1$}
  & \multicolumn{2}{c}{$T = 10$}
  & \multicolumn{2}{c}{$T = 100$}
  & \multicolumn{2}{c}{$T = 1000$} \\
\cmidrule(lr){2-3}\cmidrule(lr){4-5}\cmidrule(lr){6-7}\cmidrule(lr){8-9}\cmidrule(lr){10-11}
 & Pos & Amp & Pos & Amp & Pos & Amp & Pos & Amp & Pos & Amp \\
\midrule
1 & -4.37 & -2.16 & -4.37 & -2.16 & -4.37 & -2.16 & -4.40 & -2.06 & -4.66 & -1.98 \\
2 & -3.11 &  4.01 & -3.11 &  4.01 & -3.11 &  4.01 & -2.99 &  4.75 & -3.20 &  0.83 \\
3 & -2.13 &  4.57 & -2.13 &  4.57 & -2.13 &  4.57 & -2.00 &  3.79 & -2.46 &  6.90 \\
4 &  0.30 & -3.63 &  0.30 & -3.63 &  0.30 & -3.63 &  0.26 & -3.66 &  0.38 & -3.38 \\
5 &  2.16 & -4.67 &  2.16 & -4.67 &  2.16 & -4.67 &  2.15 & -4.70 &  2.13 & -4.60 \\
6 &  3.77 &  1.06 &  3.77 &  1.06 &  3.77 &  1.06 &  3.77 &  1.05 &  3.83 &  0.99 \\
\bottomrule
\end{tabular}
\end{table}

\subsubsection{The two-dimensional case}
In the two-dimensional case, we fix the maximum moment order to \(k_{\mathrm{max}}=10\) and the terminal time to \(T=100\).  Consequently, the discretized solution may exhibit up to 
\(
\binom{12}{2} = 66
\)
support points, far more than the true distribution.  Nevertheless, these points cluster tightly around a few centers, which serve as excellent approximations to the true Dirac support locations of \(u_0^*\). To reduce this redundancy, we merge any points whose pairwise distance is below a threshold \(0.02\).  Each resulting cluster is then replaced by its amplitude-weighted barycenter, with the cluster amplitude equal to the sum of its constituent amplitudes.  This simple post-processing step substantially improves the accuracy of the recovered measure.

Figure~\ref{fig:solutions-2d} illustrates the performance of the recovered distribution \(u_0^k\) obtained by solving \eqref{pb:OC_dis} and the post-processing step.
In Figure~\ref{fig:solutions-2d}(A), the \(W_1\) error decreases initially, exhibits a sharp drop between \(k=6\) and \(k=8\), and then begins to increase again from \(k=9\).  
  Figure~\ref{fig:solutions-2d}(B) reveals that at \(k=6\), the recovered support points are broadly scattered, although the most significant locations (dark markers) are already detected. At \(k=7\), these points begin to cluster around the key positions, and by \(k=8\), the groups have merged very close to the true support points of \(u_0^*\).

\begin{figure}[htbp]
  \centering
  \begin{subfigure}[t]{0.6\textwidth}
    \centering
    \includegraphics[width=\textwidth]{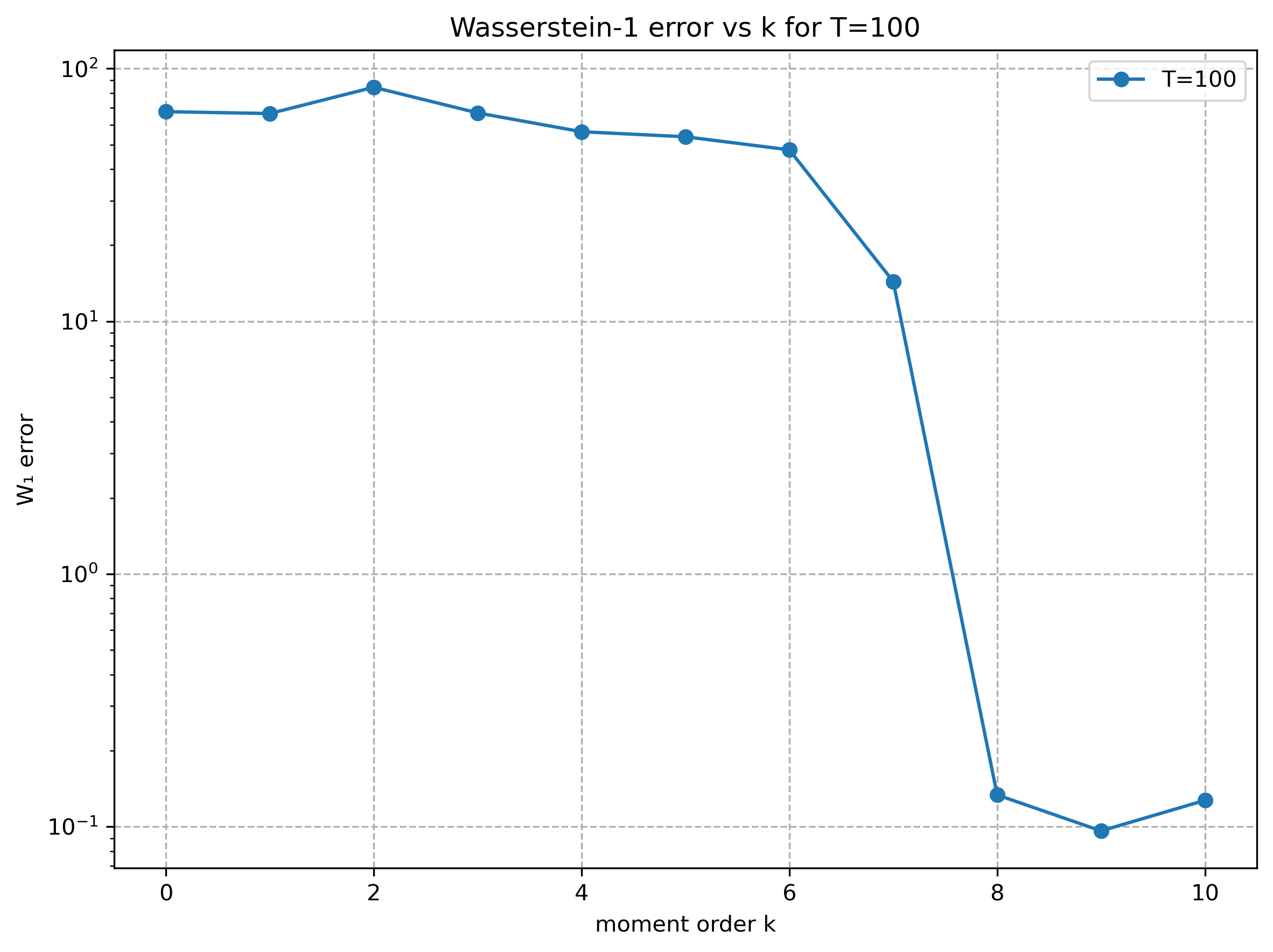}
    \caption{Evolution of \(W_1(u_0^k,u_0^*)\) by \(k\) for terminal time \(T=100\).}
  \end{subfigure}

  \vspace{2em}

  \begin{subfigure}[t]{1\textwidth}
    \centering
    \includegraphics[width=\textwidth]{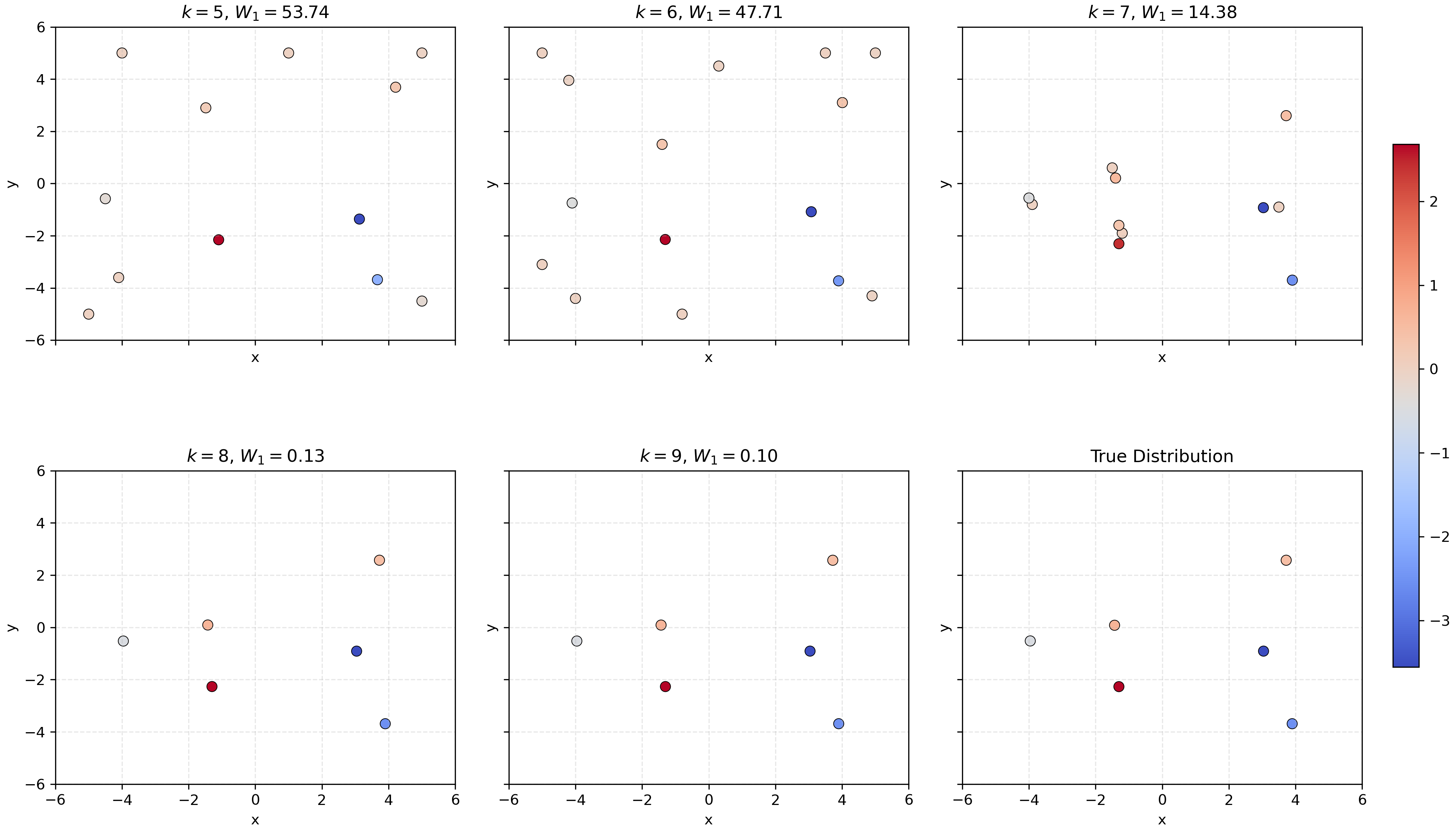}
    \caption{Solutions of \eqref{pb:OC} (after clustering) with  $k=5,\ldots,9$, $T=100$, and the true distribution.}
  \end{subfigure}

  \caption{%
    Numerical results for the 2D case: (A) Values of \(W_1(u_0^k,u_0^*)\) for \(k\in[0,10]\) and \(T=100\).%
    \quad
    (B) Recovered initial distributions \(u_0^k\) for $k=5,\ldots,9$, versus the true \(u_0^*\).%
  }
  \label{fig:solutions-2d}
\end{figure}

\section{Technical lemmas and proofs}\label{sec_proof}

\subsection{A Representer Theorem}
We first recall a representer theorem from \cite[Thm.\@ 1]{fisher1975spline}, which plays the essential role in the proof of Theorem \ref{thm:existence}.
\begin{thm}[Fisher-Jerome 75]\label{thm:representation}
Let $\Omega$ be a compact set in $\R^d$. Let $N\in \mathbb{Z}_+$ and $l_i \colon \Omega \to \mathbb{R}$ be  continuous functions for $i=1,\ldots,N$. Consider the following optimization problem:
\begin{equation}\label{pb:total_variation}
    \inf_{\mu\in \mathcal{M}(\Omega)} \|\mu\|_{\mathrm{TV}}, \qquad \mathrm{s.t. }\, \int_{\Omega} l_i(\theta) d\mu(\theta) \in I_i, \quad \mathrm{for\, }i=1,\ldots,N,
\end{equation}
where $I_i$ is a compact interval or a singleton in $\R$ for $i=1,\ldots,N$.
Assume that the feasible set of problem \eqref{pb:total_variation} is non-empty. Then, its solution set is non-empty, convex, and compact in the weak-$*$ sense. Moreover, the extreme points of the solution set of \eqref{pb:total_variation} are of the form: 
   \begin{equation*}
       \mu^* = \sum_{i=1}^N \omega_i\delta_{x_i},
   \end{equation*}
   where $\omega_i\in \R$ and $x_i\in \Omega$ for $i=1,\ldots, N$.
\end{thm}

\begin{lem}\label{lem:moment_exist}
Fix any \(k \ge 0\). Let $\Omega$ be a compact subset of $\R^d$ such that $[-r,r]^d\subseteq \Omega$ for some $r>0$. Consider the moment problem
\begin{equation}\label{pb:moment_standard}
    \inf_{\mu\in \mathcal{M}(\Omega)} \|\mu\|_{\mathrm{TV}}, \qquad \mathrm{subject to } \int_{\Omega} x^{\alpha}\, d\mu(x) = z_{\alpha}, \quad \mathrm{for all } \|\alpha\|_1 \le k,
\end{equation}
where \(\mathbf{z} = (z_\alpha)_{\|\alpha\|_1 \le k} \in \mathbb{R}^{\binom{k+d}{d}}\) is an arbitrary vector. Let \(\mu^*\) be any solution of problem \eqref{pb:moment_standard}. Then, the total variation of \(\mu^*\) satisfies the following a priori estimate:
\begin{equation}\label{eq:total_variation}
    \|\mu^*\|_{\mathrm{TV}} \le e^{\frac{2dk}{r}}\, \|\mathbf{z}\|_{\infty}.
\end{equation}
\end{lem}

\begin{proof}
Since $\Omega$ contains a nonempty interior, the collection of monomials 
\(
\{x^{\alpha} \colon \Omega \to \mathbb{R}\}_{\|\alpha\|_1 \le k}
\)
is linearly independent. Hence, the existence of a solution $\mu^*$ follows from Theorem~\ref{thm:representation}.

The dual problem of \eqref{pb:moment_standard} is given by (see \cite[Sec.~2.3]{duval2015exact})
\begin{equation}\label{pb:dual}
   \sup_{c_{\alpha}} \; \sum_{\|\alpha\|_1 \le k} z_{\alpha}\,c_{\alpha} \quad \mathrm{subject \,to} \quad \sup_{x\in \Omega}\left|\sum_{\|\alpha\|_1\le k} c_{\alpha}\,x^{\alpha}\right| \le 1.
\end{equation}
By the Fenchel--Rockafellar theorem, strong duality holds (see \cite[Prop.~13]{duval2015exact} for details):
\[
\|\mu^*\|_{\mathrm{TV}} = \mathrm{val}\eqref{pb:moment_standard} = \mathrm{val}\eqref{pb:dual}.
\]
Let $\mathbf{c}^* = (c^*_{\alpha})_{\|\alpha\|_1\le k}$ be a solution of \eqref{pb:dual}. By H\"older's inequality, we have
\[
\mathrm{val}\eqref{pb:dual} \le \|\mathbf{c}^*\|_1 \, \|\mathbf{z}\|_{\infty}.
\]
Moreover,
\[
\sup_{x\in \Omega}\left|\sum_{\|\alpha\|_1\le k} c^*_{\alpha}\,x^{\alpha}\right| \le 1.
\]
Since $\Omega$ contains $[-r,r]^d$, it follows from Lemma~\ref{lem:poly_coef} that
\[
\|\mathbf{c}^*\|_1 \le e^{\frac{2dk}{r}}.
\]
Therefore, the conclusion follows.
\end{proof}

\subsection{Results from approximation theory}

\begin{lem}[Jackson's Theorem in $\R^d$ \cite{newman1964jackson}]\label{lem:Jackson}
   For any \(L\)-Lipschitz function \( v \) defined on \( [-1,1]^d \) and any $k\geq 0$, there exists a multivariable polynomial \( p_k \) of degree less than \( k \) such that
\[
\|v - p_k\|_{\mathcal{C}([-1,1]^d)} \leq \frac{C_d\, L}{2k},
\]
where \( C_d \) is a constant depending only on the dimension \( d \).
\end{lem}

\begin{cor}\label{cor:approx_poly}
Let $\Omega= [-R,R]^d$ for some $R>0$. For any function \( v \in \mathcal{C}(\Omega) \) satisfying \(\mathrm{Lip}(v) \leq 1\) and any $k\geq 0$, there exists a multivariable polynomial \( p_k \) of degree less than \( k \) such that:
\[
\|v - p_k\|_{\mathcal{C}(\Omega)} \leq \frac{C_d \, R}{2k},
\]
where \( C_d \) is a constant depending only on the dimension \( d \).
\end{cor}
\begin{proof}
  It suffices to apply Lemma~\ref{lem:Jackson} to the rescaled function \(\bar{v}\colon [-1,1]^d\to \R, \, x \mapsto 
  v (R\, x )\).
\end{proof}

\begin{lem}\label{lem:poly_coef}
    Let \( f \) be any multivariable polynomial of degree less than \( k \).
    Assume that for $r>0$, \( |f(x)| \leq 1 \) for all \(x \in [-r,r]^d\). Then, the following holds:
\[
\sum_{\|\alpha\|_1\leq k }|c_\alpha| \leq e^{\frac{2dk}{r}},
\]
where \( c_\alpha \) are the coefficients of \( f \).
\end{lem}

\begin{proof}
  Fix any multi-index $\alpha = (\alpha_1,\dots,\alpha_d) \in \mathbb{Z}_+^d$ with $\|\alpha\|_1 \le k$. Recall that the Taylor coefficient of $f$ at $0$ corresponding to $\alpha$ is given by
\[
c_{\alpha} = \frac{1}{\alpha!}\,\frac{\partial^\alpha f(0)}{\partial x^\alpha}
= \frac{1}{\alpha!}\,\frac{\partial^{\alpha_1}\partial^{\alpha_2}\cdots \partial^{\alpha_d} f(0)}{\partial x_1^{\alpha_1}\partial x_2^{\alpha_2}\cdots \partial x_d^{\alpha_d}}.
\]

\medskip

\noindent \textbf{Step 1} (Reduction to a univariate polynomial). 
Fix any $(x_1,\dots,x_{d-1})\in [-r,r]^{d-1}$ and consider $f$ as a univariate function in the variable $x_d$. Since 
\[
|f(x)| \le 1 \quad \forall\, x \in [-r,r]^d,
\]
we apply Bernstein's inequality for higher derivatives (see, e.g., \cite[Sec.~5.2.E.5]{borwein2012polynomials}) to deduce that for any fixed $(x_1,\dots,x_{d-1})$ 
\[
\left|\frac{\partial^{\alpha_d} f}{\partial x_d^{\alpha_d}}(x_1,x_2,\dots,x_{d-1},0)\right| \le \left(\frac{2\alpha_d}{r}\right)^{\alpha_d}.
\]

\noindent \textbf{Step 2} (Induction on the number of variables).
Now, fix $x_d=0$, and define
\[
g(x_1,\dots,x_{d-1}) = \frac{\partial^{\alpha_d} f}{\partial x_d^{\alpha_d}}(x_1,\dots,x_{d-1},0).
\]
Then $g$ is a polynomial in $(x_1,\dots,x_{d-1})$ of total degree at most $k-\alpha_d$ and
\[ \|g\|_{\mathcal{C}([-r,r]^{d-1})} \leq \left(\frac{2\alpha_d}{r}\right)^{\alpha_d}.\]
Applying the same argument as in Step 1, we 
\[
\left| \frac{\partial^{\alpha_{d-1} } g}{\partial x_{d-1}^{\alpha_{d-1}}} (x_1,\ldots, x_{d-2},0)\right| \leq \left(\frac{2k}{r}\right)^{\alpha_d} \left(\frac{2(k-\alpha_d)}{r}\right)^{\alpha_{d-1}}.
\]
Proceeding inductively over the remaining variables yields:
\[
|c_{\alpha}| \le \frac{(2k)^{\alpha_d} (2(k-\alpha_d))^{\alpha_{d-1}} \cdots \Bigl(2\Bigl(k-\sum_{j=2}^d \alpha_j\Bigr)\Bigr)^{\alpha_1}}{ r^{\|\alpha\|_1}\, \alpha_1! \, \alpha_2! \cdots \alpha_d!}.
\]
It follows that
\[
|c_{\alpha}| \le \frac{(2k)^{\alpha_1+\alpha_2+\cdots+\alpha_d}}{r^{\|\alpha\|_1}\,\alpha!} = \frac{1}{\alpha!} \left(\frac{2k}{r} \right)^{\|\alpha\|_1}.
\]

\medskip

\noindent \textbf{Step 3} (Summing over all multi-indices). 
Summing these estimates over all $\alpha$ with $\|\alpha\|_1 \le k$, we have
\[
\sum_{\|\alpha\|_1 \le k} |c_\alpha| \le \sum_{\|\alpha\|_1 \le k}\frac{1}{\alpha!} \left(\frac{2k}{r} \right)^{\|\alpha\|_1}= \sum_{j=0}^k \sum_{\|\alpha\|_1 = j} \frac{1}{\alpha!}\left(\frac{2k}{r} \right)^{j}.
\]
A standard combinatorial identity shows that
\[
\sum_{\|\alpha\|_1 = j} \frac{1}{\alpha!} = \frac{d^j}{j!},
\]
so that
\[
\sum_{\|\alpha\|_1 \le k} |c_\alpha| = \sum_{j=0}^k \frac{1}{j!}\left(\frac{2dk}{r} \right)^{j} \le e^{\frac{2dk}{r}}.
\]
Thus, the proof is complete.
\end{proof}

\subsection{Technical lemmas and proofs}
\begin{lem}\label{lem:Stirling}
For any $k \geq 1$ and $T > 0$, we have
\[
\sum_{j=0}^{\lfloor k/2 \rfloor} \frac{k^j (k-1)^j}{j!} \, T^j 
\leq 
 \sqrt{\frac{k}{\pi}} \, \exp\left( k + \frac{k}{2} \ln k \right)
\max \left\{ T^{\lfloor k/2 \rfloor}, \, 1 \right\}.
\]
\end{lem}

\begin{proof}
By H\"older's inequality, we have
\[
\sum_{j=0}^{\lfloor k/2 \rfloor} \frac{k^j (k-1)^j}{j!} \, T^j
\leq 
\max\left\{ T^{\lfloor k/2 \rfloor}, \, 1 \right\} \sum_{j=0}^{\lfloor k/2 \rfloor} \frac{k^j (k-1)^j}{j!}.
\]
It remains to bound the sum
\[
S \triangleq \sum_{j=0}^{\lfloor k/2 \rfloor} \frac{k^j (k-1)^j}{j!}.
\]
Let us first consider the case where $k$ is even. Let \( k = 2m \). Then
\[
S \leq 1 + m \cdot \frac{k^m (k-1)^m}{m!}.
\]
Using Stirling's lower bound
\[
n! \geq \sqrt{2\pi n} \left( \frac{n}{e} \right)^n, \quad \mathrm{for\,} n \in \mathbb{Z}_+,
\]
we obtain
\[
\frac{k^m (k-1)^m}{m!}
\leq 
\frac{k^m (k-1)^m}{\sqrt{2\pi m} \left( \frac{m}{e} \right)^m }
=
\frac{1}{\sqrt{2\pi m}} \left( \frac{e^2 k (k-1)}{m^2} \right)^m.
\]
Since \( m = k/2 \), we have
\[
\frac{k^m (k-1)^m}{m!}
\leq 
\frac{1}{\sqrt{\pi k}} (2e(k-1))^{k/2}
=
\frac{1}{\sqrt{\pi k}} \exp\left( \frac{k}{2} \ln(2e(k-1)) \right).
\]
Therefore,
\[
S \leq 1 + \frac{1}{2} \sqrt{\frac{k}{\pi}} \exp\left( k + \frac{k}{2} \ln(k - 1) \right), \quad \mathrm{for\, even\, } k.
\]
Then for the odd case \( k = 2m + 1 \),  a similar argument gives that for odd $k\geq 3$,
\[
S \leq 1 + m \cdot \frac{k^m (k-1)^m}{m!}
\leq 1 + \frac{1}{2} \sqrt{\frac{k-1}{\pi}} \exp\left( k - 1 + \frac{k-1}{2} \ln k \right).
\]
In both cases, we can bound uniformly for all \( k \geq 2 \) by
\[
S \leq 1 + \frac{1}{2} \sqrt{\frac{k}{\pi}} \exp\left( k + \frac{k}{2} \ln k \right) \leq \sqrt{\frac{k}{\pi}} \exp\left( k + \frac{k}{2} \ln k \right).
\]
One easily checks that the inequality also holds for \( k = 1 \). Therefore, the desired result follows.
\end{proof}

\begin{proof}(Proof of Lemma~\ref{lm:quadrature}).
Let $u^{\mathrm{obs}}$ denote the noisy observation of $u(\cdot,T)$ at time $T$, and let
\[
  z_i = (z_{i_1},\dots,z_{i_d}), 
  \qquad 
  i=(i_1,\dots,i_d)\in I_n \triangleq \{1,\dots,n\}^d,
\]
be the tensorized Gauss--Hermite nodes.  
The noise assumption gives
\[
  \bigl|u^{\mathrm{obs}}(2\sqrt{T}\,z_i)-u(2\sqrt{T}\,z_i,T)\bigr|
  \le \eta,
  \qquad \forall\, i\in I_n.
\]
Recall
\(
  g_\sigma(z) \triangleq z^\alpha\,u(2\sqrt{T}\,z,T).
\)
Then for each $i\in I_n$,
\[
  \bigl|g_\sigma(z_i)-z_i^\alpha u^{\mathrm{obs}}(2\sqrt{T}\,z_i)\bigr|
  \le |z_i^\alpha|\,\eta.
\]
By \eqref{eq:Gauss-1}–\eqref{eq:Gauss-3},
\begin{equation*}
    \frac{M_\alpha(T)-M_\alpha^n(T)}
       {(2\sqrt{T})^{\|\alpha\|_1+d}} =  \int_{\R^d} g_\sigma(z)e^{-\|z\|^2}\,\mathrm{d}z-
  \sum_{i\in I_n} c_i z_i^\alpha u^{\mathrm{obs}}(2\sqrt{T}\,z_i)=
  \gamma_1 + \gamma_2,
\end{equation*}
where $c_i=\prod_{k=1}^d \omega_{i_k}$ and $\omega_{i_k}$ are the 1D Gauss--Hermite weights,
\begin{equation*}
     \gamma_1 \triangleq \int_{\R^d} g_\sigma(z)e^{-\|z\|^2}\,\mathrm{d}z-  \sum_{i\in I_n} c_i g_\sigma(z_i),
\end{equation*}
and 
\begin{equation*}
     \gamma_2 \triangleq \sum_{i\in I_n} c_i \left(g_\sigma(z_i)-z_i^\alpha u^{\mathrm{obs}}(2\sqrt{T}\,z_i) \right).
\end{equation*}
We estimate $\gamma_1$ and $\gamma_2$ separately.

\medskip
\noindent\textbf{Step 1 (Estimate of $\gamma_1$).} 
For clarity, we first treat the one–dimensional case $d=1$.
Recall that
\[
  g_\sigma(z) = z^\alpha u(\sigma z,T)e^{|z|^2} 
  = z^\alpha e^{|z|^2} G_T*u_0^*(\sigma z),
  \qquad \sigma = 2\sqrt{T}.
\]
We can extend $g_\sigma$ to an entire function on $\mathbb{C}$. Moreover, for every
$\kappa>0$ there exists $A_\kappa>0$ such that
\begin{equation}\label{eq:gsigma-growth}
  |g_\sigma(z)| \;\le\; A_\kappa\,e^{\kappa|z|^2},
  \qquad \forall z\in\mathbb{C}.
\end{equation}
(Indeed, from the convolution formula and the compact support of $m_0^*$ one
obtains at most exponential growth in $|z|$, which can be absorbed into
$e^{\kappa|z|^2}$ for any fixed $\kappa>0$.)

The standard Gauss-type quadrature error formula states that there exists $\xi$ 
in the convex hull of the quadrature nodes such that
\[
  \int_\R g_\sigma(x)w(x)\,dx - Q_n(g_\sigma)
  =
  \frac{g_\sigma^{(2n)}(\xi)}{(2n)!}
  \int_\R p_n(x)^2 w(x)\,dx,
\]
where $p_n$ is the monic orthogonal polynomial of degree $n$ with respect to the weight
$w$ and $Q_n$ is the quadrature rule. 

For Gauss--Hermite, one has $w(x)=e^{-x^2}$, $p_n(x)=2^{-n}H_n(x)$, and
\[
  \int_\R p_n(x)^2 e^{-x^2}\,dx
  = \sqrt{\pi}\,2^{-n}n!, \quad \mathrm{and}\quad |\xi|\le C_0\sqrt{n},
\]
for some $C_0>0$ independent of $n$. 

We now bound $g_\sigma^{(2n)}(\xi)$ using Cauchy’s estimate with a radius that
depends on $n$. Fix $\kappa>0$ (to be chosen later) and let $A_\kappa$
be as in \eqref{eq:gsigma-growth}. For any $r>0$,
\[
  |g_\sigma^{(2n)}(\xi)|
  \;\le\;
  \frac{(2n)!}{r^{2n}}\,
  \max_{|z-\xi|=r} |g_\sigma(z)|
  \;\le\;
  \frac{(2n)!}{r^{2n}}\,
  A_\kappa e^{\kappa(|\xi|+r)^2}.
\]
Choose $r = C_1\sqrt{n}$ with some constant $C_1>0$.
Using $|\xi|\le C_0\sqrt{n}$, we obtain
\[
  |g_\sigma^{(2n)}(\xi)|
  \;\le\;
  (2n)!\,A_\kappa
  (C_1^2 n)^{-n} \exp\!\big(\kappa(C_0+C_1)^2 n\big).
\]
Hence
\[
  |\gamma_1|
  \le
  A_\kappa\sqrt{\pi}\,
  (C_1^2 n)^{-n}\exp\!\big(\kappa(C_0+C_1)^2 n\big)\,
  2^{-n}n!.
\]
Using Stirling’s estimate $n!\le C n^n e^{-n}$, we get
\[
  |\gamma_1|
  \le
  C'
  \left(
    \frac{\exp(\kappa(C_0+C_1)^2 - 1)}{2C_1^2}
  \right)^n,
\]
for some constant $C'$ independent of $n$. Now we can choose
$\kappa>0$ and $C_1>0$ such that
\[
\rho\triangleq  \frac{\exp(\kappa(C_0+C_1)^2 - 1)}{2C_1^2} \;<\; 1.
\]
With that choice we obtain
\[
  |\gamma_1|
  \;\le\;
  C' \rho^{n}
\]
for some $0<\rho<1$ independent of $n$. This proves the desired exponential
decay for $d=1$.

In $d$ dimensions, the tensor Gauss--Hermite rule is the product of the
one–dimensional rules and $g_\sigma$ is entire with the growth
\eqref{eq:gsigma-growth} on $\mathbb{C}^d$. Applying the above one–dimensional
argument iteratively in each coordinate yields the desired estimate.

\medskip
\medskip\noindent
\textbf{Step 2 (Estimate of $\gamma_2$).}
By the formula of $\gamma_2$, we have
\[
  |\gamma_2|
  \le \eta \sum_{i\in I_n} c_i\,|z_i^\alpha|.
\]
By the convergence of the tensor Gauss--Hermite rule for polynomially growing functions, the discrete Gaussian moments satisfy
\[
  \lim_{n\to +\infty} \sum_{i\in I_n} c_i\,|z_i^\alpha| = \int_{\mathbb{R}^d} |z^\alpha| e^{-\|z\|^2}\,\mathrm{d}z < \infty.
\]
Hence there exists a constant
$C_{\|\alpha\|_1,d}>0$, independent of $n$, such that
\[
  \sum_{i\in I_n} c_i\,|z_i^\alpha| \le C_{\|\alpha\|_1,d},
\]
and therefore
\[
  |\gamma_2| \le C_{\|\alpha\|_1,d}\,\eta.
\]

\medskip
Combining the bounds on $\gamma_1$ and $\gamma_2$ and rescaling back to
$M_\alpha(T)$ yields the stated estimate.
\end{proof}

\section{Conclusion and perspectives}
In this article, we introduced a moment-based approach to the inverse source identification problem for the heat equation, a classically ill-posed problem. The key idea is to reformulate the original problem via a moment system, which mitigates the time-related ill-posedness of the backward heat equation.
The proposed scheme proceeds in two main steps: 
\begin{enumerate}
    \item inversion of the moment system to recover the moments of the initial distribution from noisy measurements;
    \item reconstruction of the initial source by minimizing the total variation under the recovered moment constraints.  
\end{enumerate}
We further established an error estimate for the reconstructed distribution in the Kantorovich distance, which depends on the moment order, terminal time, and observation noise.

From an algorithmic perspective, we detailed how moments at terminal time can be estimated using quadrature techniques, and discussed the discretization and numerical solution of the resulting constrained optimization problem. Numerical experiments demonstrate the effectiveness of our approach. In particular, our method achieves robust recovery of the initial data for terminal times up to \(T = 100\), significantly improving over existing benchmarks in the literature, which are typically limited to very small times (e.g., \(T \approx 0.01\)).

\medskip

\textit{Perspectives.} Here are some directions for future work related to this article:

\begin{enumerate}
  \item (Heat equation with boundary conditions). 
    When the heat equation \eqref{eq:heat} is posed with boundary conditions (for example, the Dirichlet condition \(u(x,t)=0\) on \(\partial\Omega\)), the ODE system \eqref{eq:ODE_moments} no longer accurately captures the evolution of the moments.  
    One alternative is to replace the raw moments by the modal coefficients
    \[
      a_m(t) \;=\;\int_{\Omega}u(x,t)\,\varphi_m(x)\,\mathrm{d}x,
    \]
    where \(\{\varphi_m\}\) are the Dirichlet eigenfunctions on \(\Omega\).  These coefficients satisfy a closed (though non-nilpotent) ODE system, but its inversion is typically less stable in time than the boundary-free moment method we have developed.  A more pragmatic approach is simply to apply our moment method while ignoring the boundary conditions.  For short time horizons (\(T\) small), the influence of the boundaries is mild, and this ``unadjusted" method can still yield accurate numerical results.

   \item(Advection-Diffusion Dynamics). 
A natural extension of the heat equation is the advection-diffusion process: 
\[
\left\{
\begin{array}{ll}
\displaystyle \partial_t m - \mathrm{div}\,\bigl(D\,\nabla m - v\,m\bigr) = 0, 
& \quad \mathrm{for}\ (x,t) \in \mathbb{R}^n \times \mathbb{R}_+, \\[6pt]
\displaystyle m(0,x) = m_0(x), 
& \quad \mathrm{for}\ x \in \mathbb{R}^n.
\end{array}
\right.
\]
where \(D\) is a positive-definite diffusion matrix and \(v\) is the advection (drift) vector.  The density \(m(t,x)\) describes the law of the stochastic process
\[
\left\{
\begin{array}{l}
\displaystyle \mathrm{d}X_t = v\, \mathrm{d}t + \sqrt{2}\, D^{1/2}\, \mathrm{d}W_t, \\[0.8em]
\displaystyle X_0 \sim m_0,
\end{array}
\right.
\]
with \(\{W_t\}_{t\ge0}\) a standard Brownian motion.  The initial source identification problem seeks to recover the distribution \(m_0\) by some information of \(m(T)\).  Our moment method is especially effective when \(D\) and \(v\) are spatially invariant and piecewise constant in time. In this setting, the resulting system of moment ODEs is closed and nilpotent.
For space-dependent $D$ and $v$,  the problem becomes significantly more complex and remains a subject for future investigation.
\end{enumerate}

\end{document}